\documentclass[reqno,11pt]{amsart}
\numberwithin{equation}{section}
\usepackage{amsmath}
\usepackage{esint} % permette di fare l'integrale tagliato
\usepackage{amsthm}
\usepackage{epsfig}
\usepackage{psfrag}
\usepackage{graphicx}
\usepackage{graphpap,latexsym,epsf}
\usepackage{color}
\usepackage{amssymb,mathrsfs,enumerate}
\usepackage{mathtools}
\definecolor{citegreen}{rgb}{0,0.6,0}
\definecolor{refred}{rgb}{0.8,0,0}
\usepackage[%
colorlinks, 
citecolor=citegreen, 
linkcolor=refred]{hyperref}
\newcommand{\HH}{\mathbb{H}}
\newcommand{\N}{\mathbb{N}}

\newcommand{\Sph}{\mathbb{S}}
\newcommand{\R}{\mathbb{R}}

\def\HHH{{\rm H}}
\def\RRR{{\mathrm R}}

\def\a{\alpha}
\def\b{\beta}

\newcommand{\pa}{\partial}

\newcommand{\ep}{\varepsilon}
\newcommand{\rmd}{{\rm d}}

\newcommand{\go}{g}

\newcommand{\Ric}{{\rm Ric}}

\newcommand{\D}{{\rm D}}
\newcommand{\DD}{{\rm D}^2}
\newcommand{\De}{\Delta}

\newcommand{\umax}{u_{\rm max}}
\newcommand{\umin}{u_{\rm min}}
\newcommand{\mmax}{m_{\rm max}}

\newcommand{\na}{\nabla}

\mathchardef\emptyset="001F

\definecolor{vgreen}{rgb}{0.1,0.5,0.2}
\definecolor{viola}{RGB}{85,26,139}
\newtheorem{theorem}{Theorem}[section]
\newtheorem{remark}{Remark}
\newtheorem{corollary}[theorem]{Corollary}
\newtheorem{definition}{Definition}
\newtheorem{proposition}[theorem]{Proposition}

\newtheorem{lemma}[theorem]{Lemma}

\newtheorem{defapp}{Definition}[section]
\begin{document}

\hyphenation{ma-ni-fold}

\title[On the mass of static metrics with positive cosmological constant -- I
]{On the mass of static metrics \\ with positive cosmological constant -- I}

\author[S.~Borghini]{Stefano Borghini}
\address{S.~Borghini, Universit\`a degli Studi di Trento,
via Sommarive 14, 38123 Povo (TN), Italy 
and Scuola Normale Superiore di Pisa,
Piazza Cavalieri 7, 56126 Pisa, Italy}
\email{stefano.borghini@unitn.it}

%\author[P.~Chru\'sciel]{Piotr Chru\'sciel}
%\address{P.~Chru\'sciel, University of Vienna, 
%Boltzmanngasse 5, A 1090 Vienna, Austria}
%\email{Piotr.Chrusciel@univie.ac.at}

\author[L.~Mazzieri]{Lorenzo Mazzieri}
\address{L.~Mazzieri, Universit\`a degli Studi di Trento,
via Sommarive 14, 38123 Povo (TN), Italy}
\email{lorenzo.mazzieri@unitn.it}

% \thanks{}

\begin{abstract} 

In this paper we propose and discuss a notion of mass for compact static metrics with positive cosmological constant.
As a consequence, we characterise the de Sitter solution as the only static vacuum metric with zero mass. Finally, we show 
how to adapt our analysis to the case of negative cosmological constant, leading to a uniqueness theorem for the Anti de Sitter spacetime.

%that our method extends with minimal modifications to the case of a negative cosmological constant, leading to a uniqueness theorem for the Anti de Sitter spacetime.
%Finally, exploiting some particular features of our formalism, we show how the same analysis can be fruitfully employed to treat the case of negative cosmological constant, leading to a uniqueness theorem for the Anti de Sitter spacetime.

\end{abstract}

\maketitle


\smallskip
\noindent 

\date{\today}

\maketitle

%%%%%%%%%%%%%%%%%%%%%%%%%%%%%%%%%
%%%%%%%%%%%%%%%%%%%%%%%%%%%%%%%%%

\section{Introduction}
\label{sec:intro}

\subsection{Setting of the problem and preliminaries.} 
\label{sub:prelim}
In this paper we consider the classification problem for {\em static vacuum metrics} in presence of a cosmological constant. These are given by triples $(M,\go, u)$, where $(M,\go)$ is an $n$-dimensional Riemannian manifold, $n \geq 3$, with (possibly empty) smooth compact boundary $\pa M$, and $u \in {\mathscr C}^\infty  (M)$ is a smooth nonnegative function obeying the following system
\begin{equation}
\label{eq:SES}
\begin{dcases}
u\,\Ric=\DD u+ \frac{2\Lambda}{n-1}\,u\,\go, & \mbox{in } M  ,\\
\ \;\, \De u=-\frac{2\Lambda}{n-1}\, u, & \mbox{in } M ,
\end{dcases}
\end{equation}
where $\Ric$, $\D$, and $\De$ represent the Ricci tensor, the Levi-Civita connection, and the Laplace-Beltrami operator of the metric $\go$, respectively, and $\Lambda\in\R$ is a constant called {\em cosmological constant}.
If the boundary $\pa M$ is non-empty, we will always assume that it coincides with the zero level set of $u$, so that, in particular, $u$ is strictly positive in the interior of $M$. In the rest of the paper the metric $\go$ and the function $u$ will be referred to as {\em static metric} and {\em static potential}, respectively, whereas the triple $(M,\go, u)$ will be called a {\em static solution} or a {\em static triple}. The reason for this terminology relies on the physical nature of the problem. In fact, a classical computation shows that if $(M,\go,u)$ is a {\em static solution}, then the Lorentzian metric $\gamma = -u^2 \, dt \otimes dt + \go$ on $X=\R\times (M\setminus\pa M)$ satisfies the {\em vacuum Einstein field equations (with cosmological constant)}
\smallskip
\begin{equation*}
\Ric_\gamma \,- \frac{\RRR_\gamma}{2} \, \gamma \,+ \Lambda \, \gamma\, = \, 0 \, , \quad \hbox{ in \,\, $X$} \, .
\end{equation*}
%\begin{equation*}
%\Ric_\gamma \, = \, \frac{2\Lambda}{n-1}\,\gamma \,  \quad \hbox{ in \,\, $\R \times (M \setminus \pa M)$} \, .
%\end{equation*}
The converse is also true in the sense that if a Lorentzian  manifold $(X, \gamma)$ satisfies the above equation, and if in addition there exists an everywhere defined time-like killing vector field whose orthogonal distribution is integrable, then it must have a warped product structure, that is, $X$ splits as $\R \times M$ and $\gamma$ can be written as $\gamma = - u^2 \, dt\otimes dt + \go$, where $(M,\go)$ and $u \in {\mathscr C}^\infty  (M)$  are respectively a Riemannian manifold and a smooth function satisfying system~\eqref{eq:SES}.

\smallskip

Coming back to the preliminary analysis of system~\eqref{eq:SES}, we list some of the basic properties of static solutions, whose proof can be found in~\cite[Lemma~3]{Ambrozio} as well as in the other references listed below.

\begin{itemize}

\item Concerning the regularity of the function $u$, we know from~\cite{Chrusciel_1,ZumHagen} that $u$ is analytic. In particular, by the results in~\cite{Sou_Sou_1}, we have that its critical level sets are discrete.
\smallskip
\item Taking the trace of the first equation and substituting the second one, it is immediate to deduce that the scalar curvature of the metric $\go$ is constant, and more precisely it holds
\begin{equation}
\label{eq:CSC}
\RRR=2\Lambda \, .
\end{equation}
In particular, we observe that choosing a normalization for the cosmological constant corresponds to fixing a scale for the metric $\go$.
\smallskip
\item In the case where the boundary is a non-empty smooth submanifold, one has that $\pa M = \{u=0\}$ is a regular level set of $u$, and in particular it follows from the equations that it is a (possibly disconnected) totally geodesic hypersurface in $(M,\go)$. The connected components of $\pa M$ will be referred to as {\em horizons}. In the case where $\Lambda>0$ we will also distinguish between horizons of black hole type, of cylindrical type and of cosmological type (see Definition~\ref{def:horiz} in Section~\ref{sec:settingandstatement_D}). 
\smallskip
\item Finally, one has that the quantity $|\D u|$ is locally constant and positive on $\pa M$. The constant value of $|\D u|$ on a given connected component of the boundary gives rise to the important notion of {\em surface gravity}, which will play a fundamental role in the subsequent discussion. 
On this regard, it is important to notice that on one hand the equations in~\eqref{eq:SES} are invariant by rescaling of the {\em static potential}, whereas on the other hand the value of $|\D u|$ heavily depends on such a choice. Hence, in order to deal with meaningful objects, one needs to remove this ambiguity, fixing a normalization of the function $u$. This is done in different ways, depending on the sign of the cosmological constant as well as on some natural geometric assumptions.
\begin{itemize}
\smallskip
\item In the case where $\Lambda>0$ and $M$ is compact, which is the one we are more interested in, we will define the {\em surface gravity} of an horizon $\Sigma\subset\pa M$ as the quantity 
\begin{equation*}
\kappa(\Sigma) \,\, = \,\, \frac{\,\,|\D u|_{|_\Sigma}}{\max_M u} \,.
\end{equation*}
%$(|\D u|/\umax)_{|_\Sigma}$, 
This definition coincides with the one suggested in~\cite{Bou_Haw,Pap_Kan}. Of course, up to normalize $u$ in such a way that $\max_M u =1$, the above definition reduces to $\kappa(\Sigma) = |\D u|_{|_\Sigma}$.
\smallskip
\item In the case where $\Lambda=0$ and $(M,g,u)$ is asymptotically flat with bounded {\em static potential}, the {\em surface gravity} of an horizon $\Sigma \subset \pa M$ is defined as 
\begin{equation*}
\kappa(\Sigma) \,\, = \,\, \frac{\,\,|\D u|_{|_\Sigma}}{\sup_M u} \,.
\end{equation*}
Again, under the usual normalization $\sup_M u =1$, the above definition reduces to $\kappa(\Sigma) = |\D u|_{|_\Sigma}$.
%simply the constant value of $|\D u|$ on $\Sigma$.
%normalization of $u$ is usually chosen by requiring that $u(x) \to 1$, as  $|x| \to + \infty$. With this assumption, the surface gravity of an horizon $\Sigma$ is simply the constant value of $|\D u|$ on $\Sigma$,
\smallskip
\item In the case where $\Lambda<0$ and $(M, g, u)$ is a conformally compact static solution (see Definition~\ref{def:CC_A} in Appendix~\ref{sec:appB}), such that the scalar curvature $\RRR^{\pa_\infty M}$ of the metric induced by (the smooth extension of) $\overline{g} = u^{-2} g$ on the boundary at infinity $\pa_\infty M$ is constant and nonvanishing,  the {\em surface gravity} of an horizon $\Sigma \subset \pa M$ is defined as 
%{\color{red} AL DENOMINATORE HO USATO IL COMANDO frac INVECE DELLA SEMPLICE $/$, PERCH\'E LA SCRITTURA $n/2(n-2)\Lambda$ MI SEMBRAVA UN PO' AMBIGUA
\begin{equation*}
\kappa(\Sigma) \,\, = \,\,
%\sqrt{-\frac{2 \Lambda}{n}} \,\,
\frac{\,\,|\D u|_{|_\Sigma}}{\sqrt{ \left| \, \dfrac{n \, \RRR^{\pa_\infty M}
%_{u^{-2}g}
}{2(n-2)\Lambda}   \,\right|}} \,.
\end{equation*}
%}
In this case, according to~\cite[Section~VII]{Chr_Sim}, a natural normalization for the static potential is the one for which, under the above assumptions, one has that the constant value of $\big| \,  \RRR^{\pa_\infty M}
%_{u^{-2}g} 
\,\big|$ coincides with $- 2 (n-2)\Lambda/n$. Having fixed this value, the surface gravity can be computed as $\kappa(\Sigma) = |\D u|_{|_\Sigma}$.
%one can fix the normalization of the potential $u$ by prescribing the scalar curvature of the conformal infinity, see for instance~\cite[Section~VII]{Chr_Sim}. 
%{\color{red} AGGIUNGERE DETTAGLI, NON POSSO CREDERE CHE NON CI SIA UNA DEFINIZIONE SENSATA DI SURFACE GRAVITY IN QUESTO CASO In any case, when $\Lambda<0$ we will focus on solutions with empty boundary, so the notion of surface gravity will not be needed.}
\end{itemize}

\end{itemize}

\smallskip

Before proceeding, it is worth giving further comments about the notion of {\em surface gravity}. In the Newtonian case, the surface gravity of a rotationally symmetric massive body (e.g. a planet of the solar system) can be physically interpreted as the intensity of the gravitational field due the body, as it is measured by a massless observer sitting on the surface of the body. For example the Newtonian surface gravity of the Earth is given by the well known value ${\rm g} = 9.8 \, {\rm m/s^2}$, the one of Jupiter is given by $2.53 \, {\rm g}$ and the one of the Sun is given by $28.02 \, {\rm g}$. Of course, in the case of black holes, the Newtonian surface gravity is no longer a meaningful concept, since it becomes infinite when computed at the horizon, regardless of the mass of the black hole. To overcome this issue one is led to introduce the appropriate relativistic concept of {\em surface gravity} of a Killing horizon. The precise definition is recalled and discussed in Appendix~\ref{sec:appB}, but roughly speaking the idea is to consider an observer sitting far away from the horizon so that its measurement of the gravitational field must be corrected by a suitable time dilation factor, giving rise to a finite number. Such a number turns out to be related to the mass of the black hole and allows to distinguish between black holes of large mass and black holes with small mass. Thus, it can reasonably be interpreted as the {\em surface gravity} of the black hole. For further insights about the physical meaning of this concept, we refer the reader to~\cite[Section~12.5]{Wald}.
%
%\bigskip
%
%To a certain extent, the static potential $u$ is the analogue of the gravitational potential in Newtonian physics. In the Newtonian case, the surface gravity can be physically interpreted as the intensity of the gravitational field at the given horizon. In the static case, this is not true anymore, since one would expect the intensity of the gravitational field to be infinite on a black hole horizon. However, by renormalizing via a redshift factor, one obtains a finite meaningful quantity, which is precisely the surface gravity defined above (for a more precise explanation, we refer the reader to~\cite[Section~12.5]{Wald}).}

Having this in mind, it is not surprising that the behavior of $|\D u|$, either at the horizons or along the geometric ends of a {\em static solution}, can be put in relation with the  mass aspect of the solution itself. To illustrate this fact, we first consider the case of an asymptotically flat static solution $(M,\go,u)$ to~\eqref{eq:SES} (e.g. in the sense of~\cite[Definition~1]{Ago_Maz_2}) with zero cosmological constant and bounded static potential. In this situation, up to rescale $u$ in such a way that $\sup_M u =1$, one can introduce the following (Newtonian-like) notion of mass 
\begin{equation*}
m(M, \go, u) \,\, = \,\, \frac{1}{(n-2) |\Sph^{n-1}|} \int_\Sigma \left\langle \, \D u \, | \,  \nu \, \right\rangle \, \rmd \sigma \, ,
\end{equation*}
where $\Sigma$ is any closed two sided regular hypersurface enclosing the compact boundary of $M$ and $\nu$ its exterior 
%(i.e., the one pointing towards the  end of $M$) 
unit normal (see for instance~\cite[Corollary~4.2.4]{Cederbaum} and the discussion below). Using the Divergence Theorem and the fact that $\Delta u = 0$, one has that the above quantity may also be computed in the following two ways
\begin{equation}
\label{eq:static_mass_0}
m (M, \go, u) \,\, = \,\, \frac{1}{(n-2) |\Sph^{n-1}|} \int_{\pa M}  \!\!\! | \D u | \, \rmd \sigma \,\, = \,\, \frac{1}{(n-2) |\Sph^{n-1}|} \,\,\lim_{R \to +\infty}\int_{S_R} \!\! \left\langle \, \D u \, | \,  \nu \, \right\rangle \, \rmd \sigma  \,,
\end{equation}
where the set $S_R$ that appears in the rightmost hand term is a coordinates sphere of radius $R$. In other words, if $x^1,\ldots,x^n$ are asymptotically flat coordinates, then $S_R = \{ |x| = R\}$.
In particular, we notice that the last expression makes sense even if $M$ has no boundary and agrees (up to a constant factor) with the more general concept of Komar mass (or KVM mass) as it is defined in~\cite[Formula (4)]{Beig}. In the same paper, it is proven that the above definition agrees with the notion of ADM mass of the asymptotically flat manifold $(M, \go)$, which according to~\cite{Arn_Des_Mis,Bartnik}, is defined as
\begin{equation*}
m_{ADM} (M, \go) \,\, = \,\,  \frac{1}{2(n-1) |\Sph^{n-1}|} \,\,\lim_{R \to +\infty}\int_{S_R} \,  \sum_{i,j} \left( \frac{\pa g_{ij}}{\pa x^i} - \frac{\pa g_{ii}}{\pa x^j}\right)\nu^j \, \rmd \sigma  \,.
\end{equation*}
On the other hand, when the boundary of $M$ is non-empty, the first expression in~\eqref{eq:static_mass_0} is also of some interest, since it can be employed to prove that, when $\pa M$ is connected, surface gravity and mass are simply proportional to each other. Summarising the above discussion, we have that for a static vacuum asymptotically flat solution $(M,g,u)$ to~\eqref{eq:SES} with $\Lambda =0$ and compact connected boundary $\pa M$, one has that 
\begin{equation*}
\frac{|\pa M |}{(n-2)|\Sph^{n-1}|} \, \kappa (\pa M) \,\, = \,\, m_{ADM}(M,g) \, .
\end{equation*}
%illustrates the relationship between the surface gravity and the mass. In particular, when $\pa M$ is connected, one has that mass and surface gravity are simply proportional to each other. 

This kind of considerations will be extremely important for  the discussion in Section~\ref{sec:settingandstatement_D}, where, based on the concept of {\em surface gravity}, we are going to propose a definition of mass for static solutions with positive cosmological constant. In fact, unlike for the cases $\Lambda = 0$ and $\Lambda <0 $, where 
natural asymptotic conditions (i.e., asymptotical flatness and asymptotical hyperbolicity) can be imposed to the solutions in order to deal with objects for which efficient notions of mass are available and fairly well understood (see~\cite{Sch_Yau,Sch_Yau_2} and~\cite{Chr_Her,Wang_1}), there is no clear definition of mass in the case where the cosmological constant is taken to be positive, at least to the authors' knowledge.
%the most important model solutions are respectively asymptotically flat and asymptotically hyperbolic, and thus they fall into classes of manifolds for which efficient notions of mass are available and fairly well understood (see~\cite{Sch_Yau,Sch_Yau_2} and~\cite{Chr_Her,Wang_1}), there is no  clear definition of mass in the case of solutions with $\Lambda >0$, at least to the authors' knowledge.
On the other hand, looking at the literature, it can be easily checked that there is a general agreement about the fact that the mass parameter $m$ that shows up in some very special classes of explicit examples should be physically interpreted as the mass of the solution. For this reason, before proceeding, we recall in a synthetic list the most important examples of rotationally symmetric solutions to system~\eqref{eq:SES}.

\subsection{Rotationally symmetric solutions.}
\label{sub:rotsol}

For the sake of completeness, we also include the case of zero cosmological constant, whereas in the other two cases, we adopt the normalization $|\Lambda |= n(n-1)/2$.

\medskip

\noindent {\bf Solutions with $\Lambda =0$.} According to~\eqref{eq:CSC}, these solutions have constant scalar curvature equal to $\RRR \equiv 0$. We also adopt the convention
\begin{equation*}
 \sup_M u \,\, = \,\, 1 \, ,
\end{equation*}
in order to fix the normalization of the static potential $u$.

\begin{itemize}
\item \underline{Minkowski solution (Flat Space Form).} 
\begin{align}
\label{eq:M}
\nonumber 
\phantom{\qquad\qquad}M \, = \, \R^n  \, ,  \qquad \go \, = \, {d|x|\otimes d|x|} +|x|^2 g_{\Sph^{n-1}} \, , \\
u \, = \, 1 \, .\phantom{\qquad\qquad\qquad\qquad\qquad\qquad}
\end{align}
It is easy to check that the metric $\go$, which a priori is well defined only in $M\setminus \{ 0\}$, extends smoothly through the origin.
This model solution has zero mass and can be seen as the limit of the following Schwarzschild solutions, when the parameter $m \to 0^+$.

\smallskip

\item \underline{Schwarzschild solutions with mass $m>0$.}
\begin{align}
\label{eq:S}
\nonumber M \, = \, \R^n \setminus {B(0,r_0(m))}  \subset\R^n \, ,  \qquad \go \, = \, \frac{d|x|\otimes d|x|}{1- 2m |x|^{2-n}}+|x|^2 g_{\Sph^{n-1}} \, , \\
u \, = \, \sqrt{1- 2m |x|^{2-n}} \, .\phantom{\qquad\qquad\qquad\qquad\qquad}
\end{align}
Here, the so called Schwarzschild radius $r_0(m) = (2m)^{1/(n-2)}$ is the only positive solution to $1-2mr^{2-n}=0$. It is not hard to check that both  the metric $\go$ and the function $u$, which a priori are well defined only in the interior of $M$, extend smoothly up to the boundary.
\end{itemize}

\medskip

\noindent {\bf Solutions with $\Lambda = n(n-1)/2$.} According to~\eqref{eq:CSC}, these solutions are normalized to have constant scalar curvature equal to $\RRR \equiv n(n-1)$. We also adopt the notations
\begin{equation*}
 \umax \,\, = \,\, \max_M u  \qquad \hbox{and} \qquad {\rm MAX}(u)=\{p\in M \, : \, u(p)=\umax\} \, ,
\end{equation*}
for the static potential $u$. Moreover, it is convenient to set 
\begin{equation}
 \label{eq:mmax_D}
 \mmax=\sqrt{\frac{(n-2)^{n-2}}{n^n}} \, .
 \end{equation}

\begin{itemize}
\item \underline{de Sitter solution (Spherical Space Form).}
\begin{align}
\label{eq:D}
\nonumber M \, = \, \overline{B(0,1)}\subset\R^n \, ,  \qquad \go \, = \, \frac{d|x|\otimes d|x|}{1-|x|^2}+|x|^2 g_{\Sph^{n-1}} \, , \\
u \, = \, \sqrt{1-|x|^2} \, .\phantom{\qquad\qquad\qquad\qquad\qquad}
\end{align}
 It is not hard to check that both  the metric $\go$ and the function $u$, which a priori are well defined only in the interior of $M\setminus \{ 0\}$, extend smoothly up to the boundary and through the origin. This model solution can be seen as the limit of the following Schwarzschild--de Sitter solutions~\eqref{eq:SD}, when the parameter $m \to 0^+$. The de Sitter solution is such that the maximum of the potential is $\umax=1$, achieved at the origin, and it has only one connected horizon with surface gravity 
\begin{equation*}
|\D u| \,\, \equiv \,\, 1  \qquad \hbox{on} \quad \pa M \, .
\end{equation*} 
Hence, according to Definition~\ref{def:horiz} below, one has that this horizon is of cosmological type.  
%and concerning the surface gravity we have 
%\begin{equation*}
%|\D u| \,\, \equiv \,\, 1  \qquad \hbox{on} \quad \pa M \, .
%\end{equation*}  
%For further considerations, it is important to notice that we have implicitly normalized the potential function $u$ in such a way that

\smallskip

\item \underline{Schwarzschild--de Sitter solutions with mass $0<m< \mmax$.}
\begin{align}
\label{eq:SD}
\nonumber 
\phantom{\qquad}M \, = \, \overline{B(0,r_+(m))} \setminus B(0, r_-(m))  \subset\R^n \, ,  \qquad \go \, = \, \frac{d|x|\otimes d|x|}{1-|x|^2- 2m |x|^{2-n}}+|x|^2 g_{\Sph^{n-1}} \, , \\
u 
\, = \, \sqrt{1-|x|^2- 2m |x|^{2-n}}\, .
\phantom{\qquad\qquad\qquad\qquad\qquad\qquad}
\end{align}
Here $r_-(m)$ and $r_+(m)$ are the two positive solutions to $1-r^2-\!2mr^{2-n}=0$. We notice that, for $r_-(m),r_+(m)$ to be real and positive, one needs~\eqref{eq:mmax_D}. It is not hard to check that both  the metric $\go$ and the function $u$, which a priori are well defined only in the interior of $M$, extend smoothly up to the boundary. This latter has two connected components with different character 
\begin{equation*}
\pa M_+  = \,\, \{ |x| = r_+(m)\} \qquad \hbox{and} \qquad \pa M_-  = \,\, \{ |x| = r_-(m)\} \, .
\end{equation*}
In fact, it is easy to check (see formul\ae~\eqref{eq:k+} and~\eqref{eq:k-}) that the surface gravities satisfy
\begin{equation*}
\phantom{\qquad\quad\,\,}\kappa (\pa M_+) \, = \, \frac{|\D u |}{\umax} \,\, < \,\, \sqrt{n}  \qquad \hbox{on $\pa M_+$}   \qquad \hbox{and} \qquad \kappa(\pa M_-) \, = \, \frac{|\D u |}{\umax} \,\, > \,\, \sqrt{n} \qquad \hbox{on  $\pa M_-$} \, .
\end{equation*}
Hence, according to Definition~\ref{def:horiz} below, one has that $\pa M_+$ is of cosmological type, whereas $\pa M_-$ is of black hole type. We also notice that it holds
$$
\umax\,=\,\sqrt{1-\left(\frac{m}{\mmax}\right)^{\frac{2}{n}}},\qquad {\rm MAX}(u)\,=\,\left\{|x|=\big[(n-2)m\big]^{\frac{1}{n}}\right\}\,,
$$ 
and $M \setminus {\rm MAX} (u)$ has exactly two connected components: $M_+$ with boundary $\pa M_+$ and $M_-$ with boundary $\pa M_-$. According to Definition~\ref{def:horiz}, we have that $M_+$ is an outer region, whereas $M_-$ is an inner region.

\begin{figure}
	\centering
	\includegraphics[scale=0.5]{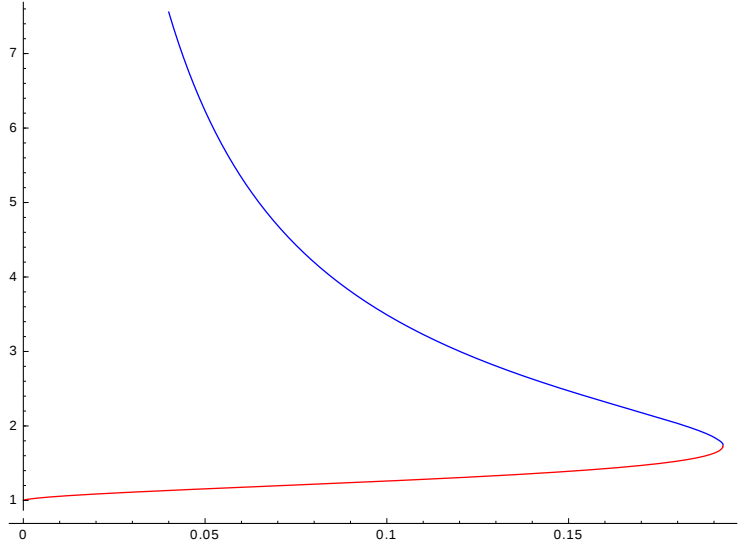}
	\caption{\small
Plot of the surface gravities $|\D u|/\umax$ of the two boundaries of the Schwarzschild--de Sitter solution~\eqref{eq:SD} as functions of the mass $m$ for $n=3$. The red line represents the surface gravity of the boundary $\pa M_+ = \{r = r_+(m)\}$, whereas the blue line represents the surface gravity of the boundary $\pa M_- = \{r = r_-(m)\}$. Notice that for $m=0$ we recover the constant value $|\D u| \equiv 1$ of the surface gravity on the (connected) cosmological horizon of the de Sitter solution~\eqref{eq:D}.
The other special situation is when $m=\mmax$. In this case the plot assigns to $\mmax= 1/(3 \sqrt{3})$ the unique value $\sqrt{3}$ achieved by the surface gravity on both the connected components of the boundary of the Nariai solution~\eqref{eq:cylsol_D}.
	} 
	\label{fig:surfacegravity_D}
\end{figure}

%to set  
%\begin{equation}
% \label{eq:mmax_D}
% \mmax=\sqrt{\frac{(n-2)^{n-2}}{n^n}} \, .
% \end{equation}
%For further considerations, it is important to notice that we have implicitly normalized the potential function $u$ in such a way that
%\begin{equation*}
%\umax \,\, = \,\, \max_M u \,\, = \,\, \sqrt{{1-\big( m/\mmax \big)^{2/n}}} \,.
%\end{equation*}

\smallskip

\item \underline{Nariai solution (Compact Round Cylinder).}
\begin{align}
\label{eq:cylsol_D}
\nonumber 
\phantom{\qquad\qquad}M \, = \, [0,\pi]\times\Sph^{n-1}\,, \qquad  \go \, = \, \frac{1}{n}\,\big[dr\otimes dr+(n-2)\,g_{\Sph^{n-1}}\big]\,, \\
u \, = \, \sin (r) \, .\phantom{\qquad\qquad\qquad\qquad\qquad\qquad\quad}
\end{align}
This model solution can be seen as the limit of the previous Schwarzschild--de Sitter solutions, when the parameter $m$ approaches $\mmax$, after an appropriate rescaling of the coordinates and of the potential $u$ (this was shown for $n=3$ in~\cite{Gin_Per} and then extended to all dimensions $n\geq 3$ in~\cite{Car_Dia_Lem}, see also~\cite{Bousso,Bou_Haw}). In this case, we have $\umax=1$ and ${\rm MAX}(u)=\{\pi/2\}\times\Sph^{n-1}$. Moreover, the boundary of $M$ has two connected components with the same constant value of the surface gravity, namely
\begin{equation*}
|\D u| \,\, \equiv \,\, \sqrt{n} \qquad \hbox{on} \quad \pa M\,=\,\{0\}\times\Sph^{n-1}\cup \{\pi\}\times\Sph^{n-1}\, .
\end{equation*}
\end{itemize}

\medskip

In Subsection~\ref{sub:surfmass}, we are going to use the above listed solutions as reference configurations in order to define the concept of {\em virtual mass} of a solution $(M, \go , u)$ to~\eqref{eq:pb_D}. For this reason it is useful to introduce since now the functions $k_+$ and $k_-$, whose graphs are plotted, for $n=3$, in Figure~\ref{fig:surfacegravity_D}. They represent the surface gravities of the model solutions as functions of the mass parameter $m$.

\smallskip

\begin{itemize}
\item The {\em outer} surface gravity function   
\begin{equation}
\label{eq:k+}
k_+  :  [\, 0, \mmax) \longrightarrow [\, 1, \sqrt{n} \, )
\end{equation}
is defined by
\begin{align*}
k_+(0) \,& = \, 1 \, , & \hbox{for $m=0$}\,, \phantom{\qquad\quad}  \\
k_+(m) \,&=\,
\sqrt{\frac{r_+^2(m)\left[1-(n-2)mr_+^{-n}(m)\right]^2}{1-\left( m / {\mmax}\right)^{2/n}}}\,, & \hspace{-1.5cm}\hbox{if $0<m<\mmax$} \, ,
\end{align*}
where $r_+(m)$ is the largest positive solution to $1-r^2-2mr^{2-n}=0$. Loosely speaking, $k_+(m)$ is nothing but the constant value of $|\D u|/\umax$ at $\{ |x| = r_+(m) \}$ for the Schwarzschild--de Sitter solution with mass parameter equal to $m$. We also observe that $k_+$ is continuous, strictly increasing and $k_+(m) \to \sqrt{n}$, as $m \to \mmax^-$.

\smallskip

\item The {\em inner} surface gravity function   
\begin{equation}
\label{eq:k-}
k_-  :  ( 0, \mmax \,] \longrightarrow [\, \sqrt{n}, +\infty \, )
\end{equation}
is defined by
\begin{align*}
k_-(\mmax) \,& = \, \sqrt{n} \, , & \hbox{for $m=\mmax$}\,, \phantom{\quad}  \\
k_-(m) \,&=\,
\sqrt{\frac{r_-^2(m)\left[1-(n-2)mr_-^{-n}(m)\right]^2}{1-\left( m / {\mmax}\right)^{2/n}}}\,, & \hspace{-1.5cm}\hbox{if $0<m<\mmax$} \, ,
\end{align*}
where $r_-(m)$ is the smallest positive solution to $1-r^2-2mr^{2-n}=0$. Loosely speaking, $k_-(m)$ is nothing but the constant value of $|\D u|/\umax$ at $\{ |x| = r_-(m) \}$ for the Schwarzschild--de Sitter solution with mass parameter equal to $m$. We also observe that $k_-$ is continuous, strictly decreasing and $k_-(m) \to + \infty$, as $m \to 0^+$.

\end{itemize}

%The surface gravities $k_\pm(m)$ at $\pa M_\pm$ satisfy
%\begin{equation}
%\label{eq:surfacegravity_SD}
%\begin{split}
%k_+(m)\,&=\,\frac{r_+(m)\Big|1-(n-2)mr_+^{-n}(m)\Big|}{\umax}\,=\,
%\sqrt{\frac{r_+^2(m)\left[1-\frac{n-2}{2}\left(r_+^{-2}(m)-1\right)\right]^2}{1-\left(\frac{r_+^{n-2}(m)-r_+^n(m)}{2\mmax}\right)}}\,,
%\\
%k_-(m)\,&=\,\frac{r_-(m)\Big|1-(n-2)mr_-^{-n}(m)\Big|}{\umax}\,=\,
%\sqrt{\frac{r_-^2(m)\left[1-\frac{n-2}{2}\left(r_-^{-2}(m)-1\right)\right]^2}{1-\left(\frac{r_-^{n-2}(m)-r_-^n(m)}{2\mmax}\right)}}\,.
%\end{split}
%\end{equation}
%In particular, one can compute that, for any $m$, it holds $k_+(m)\leq\sqrt{n}$, $k_-(m)\geq \sqrt{n}$. Therefore, it can be argued that the value of the surface gravity of the boundary distinguishes an inner from an outer region. The relevance of this observation will be clear in the next Subsection, where we will proceed to define the concepts of mass and outer/inner regions on a generic solution of our problem~\eqref{eq:prob_SD}.
%
%
\medskip

\noindent {\bf Solutions with $\Lambda = -n(n-1)/2$.}  According to~\eqref{eq:CSC}, these solutions are normalized to have constant scalar curvature equal to $\RRR \equiv -n(n-1)$. For complete solutions with empty boundary, we adopt the notations
\begin{equation*}
\umin \,\, = \,\, \min_M u \, \qquad \hbox{and} \qquad {\rm MIN}(u)=\{p\in M \, : \, u(p)=\umin\} \, .
\end{equation*}
\begin{itemize}
\item \underline{Anti de Sitter solution (Hyperbolic Space Form).}
\begin{align}
\label{eq:A}
\nonumber 
\phantom{\qquad\qquad} M \, = \, \R^n \, ,  \qquad \go \, = \, \frac{d|x|\otimes d|x|}{1+|x|^2}+|x|^2 g_{\Sph^{n-1}} \, , \\
u \, = \, \sqrt{1+|x|^2} \, .\phantom{\qquad\qquad\qquad\qquad\qquad}
\end{align}
It is easy to check that the metric $\go$, which a priori is well defined only in $M\setminus \{ 0\}$, extends smoothly through the origin.
This model solution has zero mass and can be seen as the limit of the following Schwarzschild--Anti de Sitter solutions~\eqref{eq:SA}, when the parameter $m \to 0^+$.
The Anti de Sitter solution has one end, empty boundary, and the set ${\rm MIN}(u)$ consists of a single point, the origin.
Moreover, both the function $u$ and the quantity $|\D u|$ tend to $+ \infty$ as one approaches the end of the manifold and more precisely they obey the following simple relation
\begin{equation}
\label{eq:asympt_AdS_A}
\lim_{|x|\to\infty}\left(u^2-\umin^2-|\D u|^2\right)\,\,=\,\,0\,.
\end{equation}
This fact will be of some relevance for the classification results presented in Section~\ref{sec:settingandstatement_A}.
\smallskip

\item \underline{Schwarzschild--Anti de Sitter solutions with mass $m>0$.}
\begin{align}
\label{eq:SA}
\nonumber 
\phantom{\qquad}M \, = \, \R^n \setminus B(0, r_0(m))  \subset\R^n \, ,  \qquad \go \, = \, \frac{d|x|\otimes d|x|}{1+|x|^2- 2m |x|^{2-n}}+|x|^2 g_{\Sph^{n-1}} \, , \\
u \, = \, \sqrt{1+|x|^2- 2m |x|^{2-n}} \, .\phantom{\qquad\qquad\qquad\qquad\qquad\qquad}
\end{align}
Here, $r_0(m)$ is the only positive solution to $1+r^2-2mr^{2-n}=0$. It is not hard to check that both  the metric $\go$ and the function $u$, which a priori are well defined only in the interior of $M$, extend smoothly up to the boundary
$$
\pa M\,\,=\,\,\{|x|=r_0(m)\}\,.
$$
Moreover, the triple~\eqref{eq:SA} is asymptotically Anti de Sitter in the sense of Definition~\ref{def:AAdS}. In particular, the metric $u^{-2}\go$ induces the standard spherical metric $g_{\Sph^{n-1}}$ on the conformal infinity $\pa_\infty M$ (for the definition of conformal infinity, we refer again to the Appendix, below Definition~\ref{def:CC_A}). It follows that the scalar curvature $\RRR^{\pa_\infty M}$ of the metric induced by $\overline{g} = u^{-2}\go$ on $\pa_\infty M$ is constant and equal to $(n-1)(n-2)=-2(n-2)\Lambda/n$, hence, according to the normalization suggested in Subsection~\ref{sub:prelim}, the surface gravity of the horizon $\pa M$ can be computed as
$$
\kappa({\pa M}) \,\, = \,\,  |\D u|_{|_{\pa M}}\,=\,\,r_0(m)\left[1+(n-2)\,m\,r_0^{-n}(m)\right]\,.
$$
In formal analogy with~\eqref{eq:asympt_AdS_A} one has that 
the quantities $u$ and $|\D u|$ obey the following relation
\begin{equation}
\label{eq:asympt_SADS}
\lim_{|x|\to\infty}\left(u^2- \frac{ \RRR^{\pa_\infty M}%_{u^{-2}g}
}{(n-1)(n-2)} \,\, 
%\RRR^{\pa_\infty M}_{u^{-2}g_0}
-|\D u|^2\right)\,\,=\,\,0\,.
\end{equation}
This is due to the fact that the asymptotic behavior of the Schwarzschild--Anti de Sitter solution is very similar to the one of the Anti de Sitter solution~\eqref{eq:A}. However, an important distinction is that since the boundary of the Schwarzschild--Anti de Sitter solution is non-empty and it coincides with the zero level set of the static potential, one is not allowed to replace the constant $ \RRR^{\pa_\infty M}
%_{u^{-2}g}}
/(n-1)(n-2)$ with the quantity $\umin$ in the above limit.
%This breaks the scaling invariance of~\eqref{eq:asympt_AdS_A} {\color{red}FINO A CHE PUNTO L'INVARIANZA E' DAVVERO ROTTA? SIAMO SICURI CHE LA COSA NON ABBIA A CHE VEDERE CON LA FAMOSA NORMALIZZAZIONE?} and "will be crucial in the forthcoming analysis" {\color{red} COSI' SU DUE PIEDI NON MI RICORDO CHE FOSSE UN FATTO CRUCIALE, MA NON LO VOGLIO ESCLUDERE... IN OGNI CASO CONVERREBBE INDICARE LA SOTTOSEZIONE O SEZIONE IN CUI SI ESPLICA TALE CRUCIALEZZA}. 
%
We conclude by noticing that, as far as static black hole solutions with a warped product structure are considered in the case of negative cosmological constant, 
one has that the spherical topology of the cross section is not the only possible choice. We refer the reader to Subsection~\ref{sub:nonspherical_models} in the Appendix for a description of model solutions with either flat or hyperbolic cross sections as well as for some comments on their classification.
\end{itemize}

\medskip

As anticipated, the parameter $m$ that appears in the above formul\ae\ is universally interpreted as the mass of the {\em static solution} in the physical literature. In particular, the solutions with positive mass represent the basic models for static black holes. These solutions are usually listed among static vacua, since the massive bodies which are responsible for the curvature of the space can be thought as hidden beyond some connected components of the boundary of the manifolds (horizons of black hole type). On the other hand, the solutions with zero mass should be regarded as the true static vacua. Their curvature does not depend on the presence of -- possibly hidden -- matter but it is only due to the presence of a cosmological constant. For this reasons they represent the most basic solutions to~\eqref{eq:SES} and correspond to the three fundamental geometric shapes (space forms). 

The aim of the present paper is to propose a possible characterisation of the rotationally symmetric static solutions with zero mass in presence of a cosmological constant, namely the de Sitter and the Anti de Sitter solution. As it will be made precise in the following sections, we are going to prove that these are in fact the only possible solutions to system~\eqref{eq:SES} which satisfy respectively a natural bound on the {\em surface gravity}, in the $\Lambda >0$ case, and a growth condition inspired by~\eqref{eq:asympt_AdS_A}, in the $\Lambda <0$ case. For $\Lambda >0$, we will also give, in Subsection~\ref{sub:surfmass}, an interpretation of the above mentioned result in terms of a Positive Mass Statement (see Theorem~\ref{thm:PMS} below).

%provide some characterisation of static solutions with zero mass in presence of a cosmological constant, namely the de Sitter and the Anti de Sitter solution. As it will be clarified in the following subsections, we are going to ask a suitable bound on the surface gravity when $\Lambda$ is positive, whereas, for $\Lambda<0$, we will require an asymptotic behavior in line with~\eqref{eq:asympt_AdS_A}.

\begin{remark}
%Of course, before proceeding, it is worth noticing that 
Analogous characterisations for $\Lambda = 0$ have been obtained in~\cite{Case}, where it is proven that every complete solution to~\eqref{eq:SES} with empty boundary must be Ricci flat and with constant $u$. It is also interesting to observe that if one further imposes the asymptotic flatness of the solution in the sense of~\cite[Definition~1]{Ago_Maz_2}, then it is not hard to prove that the ADM mass of such a solution must vanish. By the rigidity statement in the Positive Mass Theorem (see~\cite{Sch_Yau,Sch_Yau_2}), this implies that $(M,\go)$ is isometric to the flat Euclidean space and thus $(M, \go ,u)$ coincides with the Minkowski solution.  
\end{remark}

%%%%%%%%%%%%%%%%%%%%%%%%%%%%%%%%%%%%%%%%%%%%%%%%%%%%%
%%%%%%%%%%%%%%%%%%%%%%%%%%%%%%%%%%%%%%%%%%%%%%%%%%%%%

\section{Statement of the main results} 
\label{sec:settingandstatement_D}

%%%%%%%%%%%%%%%%%%%%%%%%%%%%%%%%%%%%%%%%%%%%%%%%%%%%%
%%%%%%%%%%%%%%%%%%%%%%%%%%%%%%%%%%%%%%%%%%%%%%%%%%%%%

\subsection{Uniqueness results for the de Sitter solution.}

For static solutions with positive cosmological constant, it is physically reasonable to assume, according to the explicit examples listed in Subsection~\ref{sub:rotsol} above, that $M$ is compact with non-empty boundary. 
As usual $u$ will be strictly positive in the interior of $M$ and such that $\pa M = \{ u= 0\}$. In order to get rid of the scaling invariance of system~\eqref{eq:SES}, we adopt the same normalization for the static metric $\go$ as in the previous subsection, so that we are led to study the system
\begin{equation}
\label{eq:pb_D}
\begin{dcases}
u\,\Ric=\DD u+n\,u\,\go, & \mbox{in } M\\
\ \;\,\De u=-n\, u, & \mbox{in } M\\
\ \ \ \ \; u>0, & \mbox{in }  M \\
\ \ \ \ \; u=0, & \mbox{on } \pa M 
\end{dcases}
 \qquad  \hbox{with} \ \  M \ \hbox{ compact}\,, \ \ \hbox{and} \ \ \RRR\equiv n(n-1)\, .
\end{equation}
Our first result is the following characterization of the de Sitter solution in terms of the surface gravity of the boundary. To state the result, we recall the notation $\umax=\max_M u$ introduced in Subsection~\ref{sub:rotsol} and for a given $\Sigma \in \pi_0(\pa M)$ (i.e., for a given connected component $\Sigma \subset \pa M$ of the boundary) we let
\begin{equation}
\label{eq:sg}
\kappa (\Sigma) \,\, = \,\, \frac{\,\,|\D u|_{|_\Sigma}}{\max_M u} \, \, \in \R \, 
\end{equation}
be the surface gravity of the horizon $\Sigma$, according to the normalization proposed in Subsection~\ref{sub:prelim}.
\begin{theorem}
\label{thm:shen_D}
Let $(M,\go,u)$ be a solution to problem~\eqref{eq:pb_D}. Then 
\begin{equation*}
\max_{\Sigma \in \pi_0(\pa M)} \kappa(\Sigma) 
%\,\, = \,\,\max_{\pa M}\frac{|\D u|}{\umax} 
\,\, \geq \,\, 1  \, .
%\qquad \hbox{on} \quad \pa M \, .
\end{equation*} 
Moreover, if the equality holds, then, up to a normalization of $u$, the triple $(M,\go,u)$ is isometric to the de Sitter solution~\eqref{eq:D}.
\end{theorem}

\noindent Recalling that the de Sitter solution~\eqref{eq:D} satisfies $\kappa(\pa M) = |\D u|/\umax=|\D u| \equiv 1$ on $\pa M$, the above result implies that the de Sitter triple has the least possible surface gravity among all the solutions to problem~\eqref{eq:pb_D} with connected boundary. The proof of the above statement is an elementary argument based on the Maximum Principle and will be presented in Section~\ref{sec:maxprinc}. More precisely, what we will prove in Theorem~\ref{thm:shen_D_dritto} below is that, if
a solution $(M, \go, u)$ to~\eqref{eq:pb_D} satisfies the inequality 
\begin{equation*}
\frac{|\D u|}{\umax} \,\, \leq \,\, 1  \qquad \hbox{on} \quad \pa M \, ,
\end{equation*}
then $(M, \go, u)$ is necessarily isometric to the de Sitter solution~\eqref{eq:D}. Combining this Maximum Principle argument together with the Monotonicity Formula of Subsection~\ref{sub:monotonicity_D}, we obtain a relevant enhancement of Theorem~\ref{thm:shen_D}, whose importance will be clarified in a moment. To introduce this result, we let ${\rm MAX}(u)$ be the set where the maximum of $u$ is achieved, namely
\begin{equation*}
{\rm MAX}(u)=\{p\in M \, : \, u(p)=\umax\} \, ,
\end{equation*}
and we observe that every connected component  of $M\setminus{\rm MAX} (u)$ has non-empty intersection with $\pa M$. This follows easily from the Weak Minimum Principle and it is proven in the No Island Lemma~\ref{lem:noisole} below. Our main result in the case $\Lambda > 0$ reads:
\begin{theorem}
\label{thm:main_D}
Let $(M,\go,u)$ be a solution to problem~\eqref{eq:pb_D}, let $N$ be a 
connected component of $M\setminus{\rm MAX}(u)$, and let $\pa N = \pa M \cap N$ be the non-empty and possibly disconnected boundary portion of $\pa M$ that lies in $N$. Then
\begin{equation*}
\max_{\Sigma \in \pi_0(\pa N)} \!\kappa(\Sigma) 
%\,\, = \,\,\max_{\pa N}\frac{|\D u|}{\umax} 
\,\, \geq \,\, 1   \,.\
%\qquad \hbox{on} \quad \pa N \, .
\end{equation*} 
Moreover, if the equality holds, then, up to a normalization of $u$, the triple $(M,\go,u)$ is isometric to the de Sitter solution~\eqref{eq:D}.
\end{theorem}

\noindent In other words, Theorem~\ref{thm:main_D} is a localized version of Theorem~\ref{thm:shen_D}. In fact, what we will actually prove (see Theorem~\ref{thm:main_D_dritto} in Section~\ref{sec:monoform}) is that if on a single connected component $N$ of $M \setminus {\rm MAX}(u)$ it holds
\begin{equation*}
\frac{|\D u |}{\umax} \,\, \leq \,\, 1 \, \qquad \hbox{on} \quad \pa N \, ,
\end{equation*}
then the entire solution $(M, \go, u)$ must be isometric to the de Sitter solution, in particular the boundary $\pa M$ and the set ${\rm MAX}(u)$ are both connected {\em a posteriori}.

\subsection{Surface gravity and mass.}
\label{sub:surfmass}
We are now in the position to present an interpretation of both Theorem~\ref{thm:shen_D} and Theorem~\ref{thm:main_D} in terms of the mass aspect of a static solution $(M, \go, u)$. As already observed, the main conceptual issue lies in the fact that, unlike for asymptotically flat and asymptotically hyperbolic manifolds, there is no clear notion of mass, when the cosmological constant is positive. To overcome this difficulty, we are going to exploit some very basic relationships between surface gravity and mass in the case of static solutions. In doing this we are motivated by the exemplification given in Subsection~\ref{sub:prelim} for $\Lambda = 0$ as well as by the explicit role played by the mass parameter $m$ in the model solutions (see Subsection~\ref{sub:rotsol}). In particular these latter are used as reference configuration in the following definition of {\em virtual mass}. As it will be clear from the forthcoming discussion, it is also useful to use them in order to distinguish between the different characters of boundary components. For this reasons we give the following definitions.

\begin{definition}
\label{def:horiz}
Let $(M, \go, u)$ be a solution to problem~\eqref{eq:pb_D}. A connected component $\Sigma$ of $\pa M$ is called an {\em horizon}. An horizon is said to be:
\begin{itemize}
\item of {\em cosmological type} if: \qquad \,\ \ 
$\kappa (\Sigma)
%(|\D u|/\umax) 
\,<\, \sqrt{n}$, 
\item of {\em black hole type} if: \qquad\quad\;\;\  
$\kappa (\Sigma)
%(|\D u|/\umax) 
\,>\, \sqrt{n}$, 
\item of {\em cylindrical type} if: \qquad \quad\,\;\
$\kappa (\Sigma)
%(|\D u|/\umax) 
\,=\, \sqrt{n}$,
\end{itemize}
where $\kappa(\Sigma)$ is the {\em surface gravity} of $\Sigma$ defined in~\eqref{eq:sg}.
A connected component $N$ of $M \setminus {\rm MAX}(u)$ is called:
\begin{itemize}
\item an {\em outer region} if all of its horizons are of cosmological type, i.e., if 
$$ \max_{\Sigma \in\pi_0(\pa N)} \kappa(\Sigma) \,<\, \sqrt{n}\,,$$
\item an {\em inner region} if it has at least one horizon of black hole type, i.e., if
$$ \max_{\Sigma \in\pi_0(\pa N)} \kappa(\Sigma) \,>\, \sqrt{n}\,,$$
\item a {\em cylindrical region} if there are no black hole horizons and there is at least one cylindrical horizon, i.e., if
$$ \max_{\Sigma \in\pi_0(\pa N)} \kappa(\Sigma) \,=\, \sqrt{n}\,.$$
\end{itemize}
\end{definition}

\noindent We introduce now the concept of virtual mass of a given connected component of $M \setminus {\rm MAX}(u)$.

\begin{definition}[Virtual Mass]
\label{def:virtual_mass} Let $(M, \go, u)$ be a solution to~\eqref{eq:pb_D} and let $N$ be a connected component of $M \setminus {\rm MAX} (u)$.
% and, for a suitable $J \in \mathbb{N}$, let $\pa_1N, \ldots, \pa_JN$ be the connected components of its boundary $\pa N$.
The virtual mass of $N$ is denoted by $\mu(N, \go ,u)$ and it is defined in the following way:
\begin{itemize}
\item[(i)] If $N$ is an outer region, then we set
\begin{equation}
\mu (N,\go,u) \,\, = \,\, k_+^{-1} \left(  \max_{\pa N} \frac{|\D u |}{\umax}  \right) \, ,
\end{equation}
where $k_+$ is the outer surface gravity function defined in~\eqref{eq:k+}.
\smallskip
\item[(ii)] If $N$ is an inner region, then we set
\begin{equation}
\mu (N, \go , u) \,\, = \,\, k_-^{-1} \left(  \max_{\pa N} \frac{|\D u |}{\umax}  \right)\, ,
\end{equation}
where $k_-$ is the inner surface gravity function defined in~\eqref{eq:k-}.
\end{itemize}
\end{definition}
\noindent In other words, the virtual mass of a connected component $N$ of $M\setminus {\rm MAX} (u)$ can be thought as the mass (parameter) that on a model solution would be responsible for (the maximum of) the surface gravity measured at $\pa N$. In this sense the rotationally symmetric solutions described in Subsection~\ref{sub:rotsol} are playing here the role of reference configurations. As it is easy to check, if $(M, \go , u)$ is either the de Sitter, or the Schwarzschild--de Sitter, or the Nariai  solution, then the virtual mass coincides with the explicit mass parameter $m$ that appears in Subsection~\ref{sub:rotsol}.

It is very important to notice that the well--posedness of the above definition for outer regions is not {\em a priori} guaranteed. In fact, one would have to check that, for any given solution $(M,\go ,u)$ to~\eqref{eq:pb_D}, the quantity $\max_{\pa N} | \D u |/\umax$ lies in the domain of definition of the function $k_+^{-1}$, namely in the real interval $[1,\sqrt{n})$. 
This is the content of the following Positive Mass Statement, whose proof is a direct consequence of Theorem~\ref{thm:main_D}.

%\bigskip
%\bigskip
%\bigskip
%--------------------------------------------
%
%
%\begin{remark}
%It is important to notice that in order to guarantee the well posedness of the above definition for outer regions, one has to check that, for any given solution $(M, \go ,u)$ to~\eqref{eq:pb_D}, the quantity $\max_{\pa N} | \D u |$ lies in the range of definition of the function $k_+^{-1}$, namely in the real interval $[1,\sqrt{n})$. On the other hand, this follows from the statement of Theorem~\ref{thm:main_D_dritto} and from the definition of {\em outer region}. 
%\end{remark}
%
%As it is easy to check, if $(M, \go , u)$ is either the de Sitter, or the Schwarzschild--de Sitter, or the Nariai  solution, then the virtual mass coincides with explicit mass parameter $m$. More interestingly, the virtual mass of a solution to~\eqref{eq:pb_D} satisfy the following Positive Mass Statement, whose proof is a direct consequence of Theorem~\ref{thm:main_D_dritto} below.
%
%
%--------------------------------------------
%
%\bigskip
%\bigskip
%\bigskip
%
%

\begin{theorem}[Positive Mass Statement for Static Metrics with Positive Cosmological Constant]
\label{thm:PMS}
Let $(M,\go,u)$ be a solution to problem~\eqref{eq:pb_D}. Then, every connected component of $M\setminus{\rm MAX}(u)$ has well--defined and thus nonnegative virtual mass. Moreover, as soon as the virtual mass of some connected component vanishes, the entire solution $(M, \go,u)$ is isometric to the de Sitter solution~\eqref{eq:D}, up to a normalization of $u$.
%Let $(M,\go,u)$ be a solution to system~\eqref{eq:pb_D}
%and let $N$ be a 
%connected component of $M\setminus{\rm MAX}(u)$. Then, the virtual mass of $N$ is nonnegative
%\begin{equation*}
%\mu (N, \go , u) \,\, \geq \,\, 0 \, .
%\end{equation*} 
%Moreover, the equality is fulfilled if and only if $(M,\go,u)$ is isometric to the de Sitter solution~\eqref{eq:D}.
\end{theorem}

\noindent In order to justify the terminology employed, it is useful to put the above result in correspondence with the classical statement of the Positive Mass Theorem for asymptotically flat manifolds with nonnegative scalar curvature. In this perspective it is clear that in our context the connected components of $M \setminus {\rm MAX} (u)$ play the same role as the asymptotically flat ends of the classical situation. In fact, the virtual mass is well defined and nonnegative on every single connected component, in perfect analogy with the 
ADM mass of every single asymptotically flat end. This correspondence holds true also for the rigidity statements. In fact, as soon as the mass (either virtual or ADM) annihilates on one single piece, the whole manifold must be isometric to the model solution with zero mass (either de Sitter or Minkowski).

Another important observation comes from the fact that the above statement should be put in contrast with the so called Min-Oo conjecture, which asserts that a compact Riemannian manifold $(M^n,g)$, whose boundary is isometric to $\Sph^{n-1}$ and totally geodesic, must be isometric to the standard round hemisphere $(\Sph^n_+, g_{\Sph^n})$, provided $\RRR_g \geq n(n-1)$. For long time, this conjecture has been considered as the natural counterpart of the rigidity statement of the Positive Mass Theorem in the case of positive cosmological constant. However, it has finally been disproved in a remarkable paper~\cite{Bre_Mar_Nev} by Brendle, Marques and Neves (we refer the reader to~\cite{Bre_Mar_Nev} also for a comprehensive introduction to the partial affirmative answers to the Min-Oo conjecture). 
In contrast with this, our Positive Mass Statement seems to indicate -- at least in the case of static solutions -- a different possible approach towards the extension of the classical Positive ADM Mass Theorem to the context of positive cosmological constant. In this perspective, it would be very interesting to see if the above statement could be extended to a broader class of metrics of physical relevance, leading to a more comprehensive definition of mass. The first step in this direction would be to consider the case of stationary solutions to the Einstein Field Equations with Killing horizons, so that the concept of surface gravity is well defined (see Appendix~\ref{sec:surf_grav}). This will be the object of further investigations.

%stationary solutions with Killing horizons. 
%
%adapted to a broader class of metrics of physical relevance (e.g., stationary solution to the Einstein Field Equations). It is our opinion that this would possibly suggest a more general definition of mass.

\subsection{Area bounds.}
Further evidences in favour of the {\em virtual mass} will be presented in the forthcoming paper~\cite{Bor_Maz_2-II}, where sharp area bounds will be obtained for horizons of black hole and cosmological type, the equality case being characterised by the Schwarzschild--de Sitter solution~\eqref{eq:SD}.
%{\color{blue}
In order to anticipate these results, we discuss in this subsection a local version of the following well known integral inequality
$$
0 \,\,\, \leq \, \int_{\pa M}\!\!|\D u|\left[\RRR^{\pa M}-(n-1)(n-2)\right]\rmd\sigma\,,
$$
which was proved by Chrusciel~\cite{Chrusciel_2} (see also~\cite{Hij_Mon_Rau}) generalizing the method introduced by Boucher, Gibbons and Horowitz in~\cite{Bou_Gib_Hor}. 
%}

\begin{theorem}
\label{thm:BGH_rev}
Let $(M,\go,u)$ be a solution to problem~\eqref{eq:pb_D}, let $N$ be a 
connected component of $M\setminus{\rm MAX}(u)$, and let $\pa N = \pa M \cap N$ be the non-empty and possibly disconnected boundary portion of $\pa M$ that lies in $N$. Suppose also that $|\D u|^2\leq C(\umax-u)$ on the whole $M$, for some constant $C\in\R$. Then it holds
\begin{equation}
\label{eq:BGH_rev}
0 \,\,\, \leq \, \int_{\pa N}\!\!|\D u|\left[\RRR^{\pa N}-(n-1)(n-2)\right]\rmd\sigma\,,
\end{equation}
where $\RRR^{\pa N}$ is the scalar curvature of the metric induced by $\go$ on $\pa N$. Moreover, if the equality holds, then, up to a normalization of $u$, the triple $(M,\go,u)$ is isometric to the de Sitter solution~\eqref{eq:S}.
\end{theorem}
\noindent The proof of this result follows closely the one presented in~\cite[Section~6]{Chrusciel_2}, see Subsection~\ref{sub:BGH_rev} for the details. The technical hypotesis $|\D u|^2\leq C(\umax-u)$ is useful in order to simplify the proof, but it can be removed at the cost of some more work (in fact, in~\cite{Bor_Maz_2-II} we will prove that this hypotesis is satisfied by any solution of~\eqref{eq:pb_D}).

In order to emphasize the analogy with the forthcoming results in the case of conformally compact static solutions with negative cosmological constant (see Corollary~\ref{cor:consequences_ALAdS} below), it is useful to single out the following straightforward consequence of the above theorem.

\begin{corollary}
\label{cor:BGH_rev}
Let $(M,\go,u)$ be a solution to problem~\eqref{eq:pb_D}, let $N$ be a 
connected component of $M\setminus{\rm MAX}(u)$, and let $\pa N = \pa M \cap N$ be the non-empty and possibly disconnected boundary portion of $\pa M$ that lies in $N$. Suppose also that $|\D u|^2\leq C(\umax-u)$ on the whole $M$, for some constant $C\in\R$. Then, if the inequality
\begin{equation*}
\RRR^{\pa N}\,\,\leq\,\,(n-1)(n-2) 
\end{equation*}
holds on $\pa N$,
where $\RRR^{\pa N}$ is the scalar curvature of the metric induced by $\go$ on $\pa N$, then, up to a normalization of $u$, the triple $(M,\go,u)$ is isometric to the de Sitter solution~\eqref{eq:S}.
\end{corollary}

\noindent
If we assume that $\pa N$ is connected and orientable, and that $n=3$ in Theorem~\ref{thm:BGH_rev}, then $|\D u|$ is constant on the whole $\pa N$, and from the Gauss-Bonnet Theorem we have $\int_{\pa N}\RRR^{\pa N}\rmd\sigma=4\pi\chi(\pa N)$. Therefore, with these additional hypoteses, the thesis of Theorem~\ref{thm:BGH_rev} translates into
$$
|\pa N|\,\leq\,2\pi\,\chi(\pa N)\,,
$$
where $|\pa N|$ is the hypersurface area of $\pa N$ with respect to the metric $\go$.
In particular, $\chi(\pa N)$ has to be positive, which implies that $\pa N$ is diffeomorphic to a sphere and $\chi(\pa N)=2$. This proves the following corollary.
\begin{corollary}
\label{cor:BGH_local}
Let $(M,\go,u)$ be a $3$-dimensional orientable solution to problem~\eqref{eq:pb_D}, let $N$ be a 
connected component of $M\setminus{\rm MAX}(u)$, and suppose that $\pa N = \pa M \cap N$ is connected. Suppose also that $|\D u|^2\leq C(\umax-u)$ on the whole $M$, for some constant $C\in\R$. Then $\pa N$ is a sphere and it holds
\begin{equation}
\label{eq:BGH_local}
|\pa N|\,\,\leq\,\, 4\pi\,.
\end{equation}
Moreover, if the equality holds then, up to a normalization of $u$, the triple $(M,\go,u)$ is isometric to the de Sitter solution~\eqref{eq:S}.
\end{corollary}
\noindent This corollary is a local version of the well known Boucher-Gibbons-Horowitz inequality~\cite[formula~(3.1)]{Bou_Gib_Hor}. In the forthcoming paper~\cite{Bor_Maz_2-II}, we are going to prove stronger versions of both inequality~\eqref{eq:BGH_rev} and~\eqref{eq:BGH_local}. In particular, we will show that, if $(M,\go,u)$ is a $3$-dimensional solution to problem~\eqref{eq:pb_D}, and $N$ is an outer (respectively, inner) 
region of $M\setminus{\rm MAX}(u)$ in the sense of  Definition~\ref{def:horiz}, with connected boundary $\pa N = \pa M \cap N$ and virtual mass $m=\mu(N,\go,u)$, then it holds
\begin{equation}
\label{eq:RP}
|\pa N|\,\,\leq\,\, 4\pi r_+^2(m)\,, \quad \hbox{(respectively, $|\pa N|\,\,\leq\,\, 4\pi r^2_-(m)$\,)},
\end{equation}
where $0\leq r_-(m)\leq r_+(m)\leq 1$ are the two nonnegative solutions of $1-x^2-2m/x=0$.
The first inequality in~\eqref{eq:RP} is an area bound for cosmological horizons, whereas the second inequality in~\eqref{eq:RP} can be seen as a Riemannian Penrose-like inequality for black hole horizons. 
In order to justify the latter terminology we recall that, in the case $\Lambda=0$, the well--known $3$--dimensional Riemannian Penrose Inequality~\cite[formula~(0.4)]{Hui_Ilm} can be written as $|\Sigma|\leq 4\pi r_0^2(m)$, where $\Sigma$ is {\em any} connected component of $\pa M$, $m$ is the ADM mass and $r_0(m)=2m$ is the Schwarzschild radius.
%
%
%we recall that the cosmological horizons are the ends of the manifold $M$, while the components of $\pa M$ can be interpreted as black hole horizons. In this case, the well--known $3$--dimensional Riemannian Penrose Inequality~\cite[formula~(0.4)]{Hui_Ilm} can be written as $|\Sigma|\leq 4\pi r_0^2(m)$, where $\Sigma$ is {\em any} connected component of $\pa M$, $m$ is the ADM mass and $r_0(m)=2m$ is the Schwarzschild radius.
Starting from inequality~\eqref{eq:RP}, also a Black Hole Uniqueness Statement will be proven, provided e.g. the set ${\rm MAX}(u)$ is a two sided regular hypersurface that divides $M$ into an inner region and an outer region, whose virtual mass is controlled by the one of the inner region.

\section{\texorpdfstring{Analogous results in the case $\Lambda<0$.}{Analogous results in the case Lambda<0.}} 
\label{sec:settingandstatement_A}

\noindent In this section we discuss the case $\Lambda<0$. While this is not the main topic of this work, it is remarkable that the same techniques used in the following sections to prove Theorems~\ref{thm:shen_D} and~\ref{thm:main_D}, can be easily adapted to prove uniqueness results for the Anti de Sitter triple~\eqref{eq:A}. 
Nevertheless, the reader interested only to the case $\Lambda>0$ can skip this section entirely.

The results that we will prove in this section, namely Theorems~\ref{thm:shen_A} and~\ref{thm:main_A}, seem to be in line with the uniqueness result proved by Case in~\cite{Case}, which states that, in the case $\Lambda=0$, a complete three dimensional static metric without boundary is covered by the Minkowski triple~\eqref{eq:M} (in general, for $n\geq4$, the conclusion is that such a metric must be Ricci flat). 
We observe that, while in the case $\Lambda=0$ no hypotesis on the behavior at infinity of the solutions was required, in the case $\Lambda<0$ we cannot expect our uniqueness results to remain true without additional assumptions. In fact, the Anti de Sitter triple is not the only solution to~\eqref{eq:pb_A}. Another one is the Anti Nariai triple~\eqref{eq:cylsol_A} described in Appendix~\ref{sec:appB}, and we also point out that the existence of an infinite family of conformally compact solutions has been proven in~\cite{And_Chr_Del-I,And_Chr_Del-II} (see Subsection~\ref{sub:comp_with_other_charact} for some more details).
To rule these solutions out and obtain uniqueness statements for the Anti de Sitter triple, we suggest the following possibility. Recalling the asymptotic behaviour~\eqref{eq:asympt_AdS_A} that is expected on the model solution~\eqref{eq:A}
\begin{equation*}
%\label{eq:asympt_AdS_A}
\lim_{|x|\to\infty}\left(u^2-\umin^2-|\D u|^2\right)\,\,=\,\,0\,,
\end{equation*}
and looking at this formula as to a necessary condition, we are going to show that it also yields a fairly neat sufficient condition in order to conclude that a complete static triple with $\Lambda<0$ is isometric to the Anti de Sitter solution. Formally, this translates in the characterisation of the equality case in formul\ae~\eqref{eq:shen_A} and~\eqref{eq:main_A} below.

%In the case $\Lambda < 0$ our reference model will be the Anti de Sitter solution instead of the de Sitter solution. It is important to remark that there is a relevant topological difference between these two solutions, namely the cosmological horizon of the de Sitter triple is a compact boundary component, while the one of the Anti de Sitter triple is at infinity. As a consequence, the results that we obtain in the case $\Lambda<0$ apply to solutions with ends but {\em with empty boundary}. It follows that, in the case $\Lambda<0$, our results will not give rise to a notion of mass analogous to the one introduced in Subsection~\ref{sub:surfmass}, simply because in our hypoteses there are no horizons where we can compute the surface gravity.

\subsection{Uniqueness results for the Anti de Sitter solution.}

In analogy with the properties of the Anti de Sitter triple~\eqref{eq:A} described in Subsection~\ref{sub:rotsol}, it is natural to restrict our attention to static solutions $(M,\go,u)$ with negative cosmological constant such that the manifold $M$ is non-compact and with empty boundary. We point out that the latter assumption, which is unavoidable in the present framework, excludes {\em a priori} the family of the Schwarzschild--Anti de Sitter solutions~\eqref{eq:SA} from our treatment. 
%It is important to remark that the assumption $\pa M=\emptyset$, which is necessary for our analysis, rules out the Schwarzschild--Anti de Sitter solutions~\eqref{eq:SA}. 
For simplicity, we will also suppose that the number of ends of $M$ is finite. We recall (see for instance~\cite[Section~3.1]{Guilbault}) that the ends of $M$ are defined as the sequences $U_1\supset U_2\supset\dots$, where, for every $i\in\N$, $U_i$ is an unbounded connected component of $M\setminus K_i$ and $\{K_i\}$ is an exhaustion by compact sets of $M$. It is easy to see that the definition of end does not really depend on the choice of the exhaustion by compact sets, in the sense that there is a clear one-to-one correspondence between the ends of $M$ defined with respect to two different exhaustions. We emphasize the fact that -- in contrast with other characterisations of the Anti de Sitter solution -- we are not making any {\em a priori} assumption on the topology of the ends, as it is explained in Figure~\ref{fig:ends_A} and the discussion below.
Starting from system~\eqref{eq:SES}, and rescaling $\go$ as in Subsection~\ref{sub:rotsol}, we are led to study the following problem

\begin{equation}
\label{eq:pb_A}
\begin{dcases}
u\,\Ric=\DD u-n\,u\,\go, & \mbox{in } M\\
\ \;\,\De u=n\, u, & \mbox{in } M\\
\ \ \ \ \; u>0, & \mbox{in }  M \\
\  u(x)\rightarrow +\infty & \mbox{as } x\rightarrow \infty
\end{dcases}
 \qquad  \hbox{with} \quad \pa M\,=\,\emptyset \quad \hbox{and} \quad \RRR\equiv -n(n-1)\, .
\end{equation}

\begin{figure}
	\centering
	\includegraphics[scale=2.0]{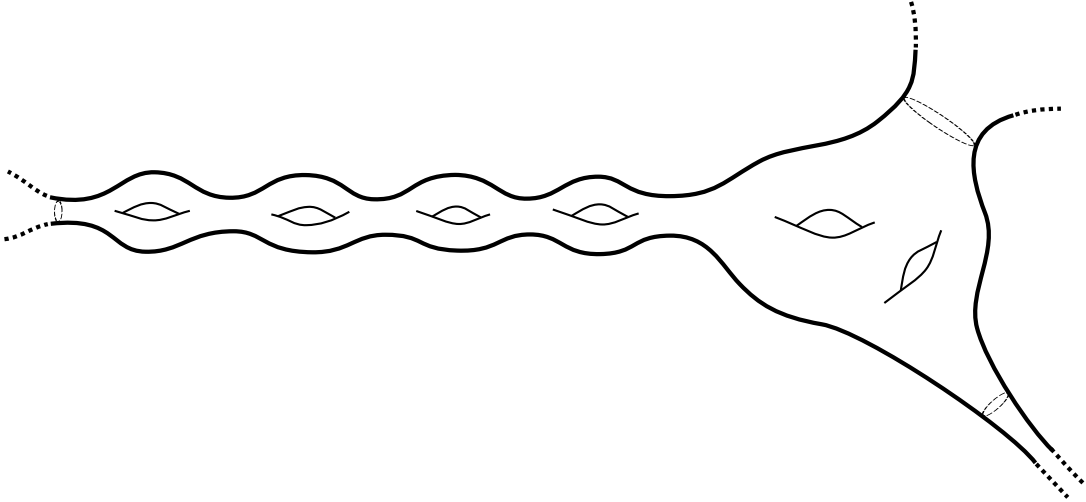}
\caption{\small
The ends of the solutions to~\eqref{eq:pb_A} are usually assumed to be diffeomorphic to a product. However, our analysis does not exclude {\em a priori} more peculiar topologies, like the end represented on the left hand side of the figure.
} 
\label{fig:ends_A}
\end{figure}

\noindent With the notation $u(x) \to + \infty$ as $x\to\infty$, we mean that, 
%the point $x$ is converging to the infinite of $M$. More precisely, the last condition in system~\eqref{eq:pb_A} means that, 
given an exhaustion of $M$ by compact sets $\{K_i\}_{i\in\N}$, we have that for any sequence of points $x_i\in M\setminus K_i$, $i\in\N$, it holds $\lim_{i\to+\infty} u(x_i)=+\infty$. Recalling the notation $\umin=\min_M u$, we are now able to state our first result in the $\Lambda <0$ case. The proof follows the same line as the one of Theorem~\ref{thm:shen_D} in the $\Lambda >0$ case.
%Adopting the techniques used in the case $\Lambda>0$ to prove Theorem~\ref{thm:shen_D}, and recalling the notation $\umin=\min_M u$, we are able to obtain the following result.

\begin{theorem}
	\label{thm:shen_A}
	Let $(M,\go,u)$ be a solution to problem~\eqref{eq:pb_A}. Then
\begin{equation}
\label{eq:shen_A}
\liminf_{x\to\infty}\left(u^2-\umin^2-|\D u|^2\right)(x)\,\leq\, 0\,.
\end{equation}
Moreover, if the equality holds then, up to a normalization of $u$, the triple $(M,\go,u)$ is isometric to the Anti de Sitter solution~\eqref{eq:A}.
\end{theorem}

\noindent 
To avoid ambiguity, we recall that inequality~\eqref{eq:shen_A} means that, taken an exhaustion of $M$ by compact sets $\{K_i\}_{i\in\N}$, we have that for any sequence of points $x_i\in M\setminus K_i$, $i\in\N$, it holds $\liminf_{i\to+\infty} (u^2-\umin^2-|\D u|^2)(x_i)\leq 0$.
We have already observed in~\eqref{eq:asympt_AdS_A} that the Anti de Sitter triple~\eqref{eq:A} is such that $u^2-\umin^2-|\D u|^2$ goes to zero as one approaches the end of the manifold. Therefore, Theorem~\ref{thm:shen_A} characterizes the Anti de Sitter triple among the solutions to~\eqref{eq:pb_A} as the one that maximises the left hand side of~\eqref{eq:shen_A}. In fact, what we will actually prove (see Theorem~\ref{thm:shen_A_dritto} below) is that the only solution to~\eqref{eq:pb_A} that satisfies
$$
\liminf_{x\to\infty}\left(u^2-\umin^2-|\D u|^2\right)(x)\,\,\geq\,\, 0\,,
$$
is the Anti de Sitter triple~\eqref{eq:A}.

We are now going to state a local version of Theorem~\ref{thm:shen_A}. To this end, we denote the set of the minima of $u$ as
$$
{\rm MIN}(u)\,=\,\{p\in M\,:\,u(p)=\umin\}\,,
$$
and we notice that any connected component $N$ of $M\setminus{\rm MIN}(u)$ must contain at least one of the ends of $M$ by the No Island Lemma~\ref{lem:noisole_A}. In particular, the $\liminf_{x \to \infty}$ in formula~\eqref{eq:main_A} below is completely justified.

 Arguing as in the case $\Lambda>0$, we obtain through a Maximum Principle and a suitable Monotonicity Formula the following analogue of Theorem~\ref{thm:main_D}.

\begin{theorem}
\label{thm:main_A}
Let $(M,\go,u)$ be a solution to problem~\eqref{eq:pb_A}, and let $N$ be a connected component of $M\setminus{\rm MIN}(u)$.
Then
	\begin{equation}
	\label{eq:main_A}
	\liminf_{x\in N,\,x\to\infty}\left(u^2-\umin^2-|\D u|^2\right)(x)\,\,\leq\,\, 0\,.
		\end{equation}
Moreover, if the equality holds then, up to a normalization of $u$, the triple $(M,\go,u)$ is isometric to the Anti de Sitter solution~\eqref{eq:A}.
\end{theorem}

\noindent Theorem~\ref{thm:main_A} is a stronger version of Theorem~\ref{thm:shen_A}, in the sense that the asymptotic behavior of the quantity in~\eqref{eq:main_A} has to be checked only along the ends of $N$. In this sense, the relation between Theorem~\ref{thm:shen_A} and  Theorem~\ref{thm:main_A} is the same as the one between Theorem~\ref{thm:shen_D} and Theorem~\ref{thm:main_D}. In the next subsection, we are going to compare Theorem~\ref{thm:main_A} with other known characterisations of the Anti de Sitter solution.

\subsection{Comparison with other known characterizations.}
\label{sub:comp_with_other_charact}

Classically, the study of static solutions with $\Lambda<0$ has been tackled by requiring some additional information on the asymptotic behavior of the triple $(M,\go,u)$. These assumptions, albeit natural, are usually very strong, in the sense that they restrict the topology of the ends as well as the asymptotic behavior of the function $u$.
The main definitions and known results are discussed in details in Subsection~\ref{sub:defi_app} of the Appendix. Here, we quickly recall them in order to draw the state of the art and put our results in perspective.

The most widely used assumption is to ask for the triple $(M,\go,u)$ to be {\em conformally compact} in the sense of Definition~\ref{def:CC_A} in Appendix~\ref{sec:appB}. This hypotesis forces $(M,\go)$ to be isometric to the interior of a compact manifold $\overline{M}_\infty=M\cup\pa_\infty M$, where $\pa_\infty M$ is the boundary of $\overline{M}_\infty$ and is called the {\em conformal infinity} of $M$. It also requires the metric $\bar g=u^{-2}\go$ to extend to the conformal infinity with some regularity. Despite this being a somewhat standard assumption, almost nothing being known without requiring it, it still imposes some strong topological and analytical {\em a priori} restrictions on a mere solution to~\eqref{eq:pb_A}.
For instance, if $n=3$, we know from~\cite{Chr_Sim} (see also Proposition~\ref{pro:CC_A}) that the conformal infinity $\pa_\infty M$ is necessarily connected, that is, $M$ has a unique end. Therefore, for $3$-dimensional {\em conformally compact} triples, Theorems~\ref{thm:shen_A} and~\ref{thm:main_A} are completely equivalent.
Nevertheless, the conformal compactness {\em per se} is not strong enough to characterize the Anti de Sitter solution. In fact, it is proven in~\cite[Theorem~1.1]{And_Chr_Del-I} that, for any Riemannian metric $\gamma$ on $\Sph^2$ with the property that the Lorentzian manifold $(\R\times\Sph^2,-dt\otimes dt+\gamma)$ has positive scalar curvature, there exists a conformally compact $3$-dimensional solution $(M,\go,u)$ of~\eqref{eq:pb_A} such that the metric induced by $\bar g=u^{-2}\go$ on the conformal infinity coincides with $\gamma$. Such a general result is not available in higher dimensions, however the existence of an infinite family of solutions to~\eqref{eq:pb_A} for any $n\geq 4$ has been proven in~\cite[Theorem~1.1]{And_Chr_Del-II}, showing for instance that any small perturbation of the standard metric on $\Sph^{n-1}$ can be realised as the metric induced on the conformal infinity of a conformally compact static solution through the usual formula.

This implies that, even in the case of conformally compact solutions, one needs to make some additional assumptions in order to prove a rigidity statement. To introduce our next result, we recall that, for conformally compact solutions, the quantity $u^2-\umin^2-|\D u|^2$ extends smoothly to a function on the whole $\overline{M}_\infty=M\cup\pa_\infty M$ and it holds (see formula~\eqref{eq:ALAdS_A} in Appendix~\ref{sec:appB})
\begin{equation}
\label{eq:ALAdS_revisited_A}
u^2-\umin^2-|\D u|^2\,=\,\frac{\RRR^{\pa_\infty M}}{(n-1)(n-2)}-\umin^2\qquad\hbox{on }\,\pa_\infty M\,,
\end{equation}
where $\RRR^{\pa_\infty M}$ is the scalar curvature of the metric induced by $\bar g$ on $\pa_\infty M$.
% To this end, we recall formula~\eqref{eq:ALAdS_A} in Appendix~\ref{sec:appB}, which tells us that, for conformally compact solutions, the quantity $u^2-\umin^2-|\D u|^2$ extends {\color{red} COME? CON CEH REGOLARITA'? SMOOTHLY? ALTRIMENTI QUALUNQUE COSA SI ESTENDE ADUNA FUNZIONE....} to a function on the whole $\overline{M}_\infty=M\cup\pa_\infty M$ and it holds
%\begin{equation}
%\label{eq:ALAdS_revisited_A}
%u^2-\umin^2-|\D u|^2\,=\,\frac{\RRR^{\pa_\infty M}}{(n-1)(n-2)}-\umin^2\qquad\hbox{on }\,\pa_\infty M\,,
%\end{equation}
%where $\RRR^{\pa_\infty M}$ is the scalar curvature of the metric induced by $\bar g$ on $\pa_\infty M$. 

In order to introduce the next result, we first fix a couple of notations. Given a connected component $N$ of $M\setminus{\rm MAX}(u)$, we denote by $\pa_\infty N$ the conformal infinity of $N$, that is, 
$$\pa_\infty N\,=\,\pa_\infty M\cap \overline{N}^{\overline{M}_\infty}\,,
$$
where $\overline{N}^{\overline{M}_\infty}$ is the closure of $N$ in $\overline{M}_\infty$.
From formula~\eqref{eq:ALAdS_revisited_A} and Theorem~\ref{thm:main_A} we deduce the following corollary, that represents the precise analogue of Corollary~\ref{cor:BGH_rev}, proven in the case $\Lambda>0$.

\begin{corollary}
\label{cor:consequences_ALAdS}
Let $(M,\go,u)$ be a conformally compact solution to problem~\eqref{eq:pb_A} and let $N$ be a connected component of $M\setminus{\rm MIN}(u)$. Suppose that the scalar curvature $\RRR^{\pa_\infty N}$ of the metric induced by $\bar g=u^{-2}\go$ on the conformal infinity 
%$\pa_\infty N$ 
of $N$ satisfies the following inequality
\begin{equation}
\label{eq:consequences_ALAdS}
\RRR^{\pa_\infty N}\,\,\geq\,\, (n-1)(n-2)\,\umin^2
\end{equation}
on the whole $\pa_\infty N$. Then, up to a normalization of $u$, the triple $(M,\go,u)$ is isometric to the Anti de Sitter solution~\eqref{eq:A}.
\end{corollary}

\noindent Imposing stronger assumptions on the asymptotics of the triple $(M,\go,u)$ leads to 
even cleaner statements. The fee for this is that the class of solutions where the uniqueness can be proven is {\em a priori} much smaller than the ones considered above. For example, if one requires the triple $(M,\go,u)$ to be {\em asymptotically Anti de Sitter} in the sense of Definition~\ref{def:AAdS}, then it is possible to conclude uniqueness as in Corollary~\ref{cor:small_wang_A} below. However, this assumption forces the conformal infinity of $M$ -- endowed with the metric induced on it by $\overline{g}= u^{-2} g$ -- to be connected and isometric to the standard sphere. In particular the quantity $\RRR^{\pa_\infty N}$ in Corollary~\ref{cor:consequences_ALAdS} is equal to $(n-1)(n-2)$ on the whole $\pa_\infty N=\pa_\infty M$.
%
%
%
%If one imposes a stronger assumption on the asymptotics, namely that the triple $(M,\go,u)$ is {\em asymptotically Anti de Sitter} in the sense of Definition~\ref{def:AAdS}, then one can get a cleaner uniqueness result. In fact, this additional hypotesis forces the conformal infinity of $M$ to be connected and isometric to the standard sphere, hence the quantity $\RRR^{\pa_\infty N}$ in Corollary~\ref{cor:consequences_ALAdS} is equal to $(n-1)(n-2)$ on the whole $\pa_\infty N=\pa_\infty M$. Therefore, the corollary above translates in the following uniqueness result for the Anti de Sitter solution.

\begin{corollary}
	\label{cor:small_wang_A}
	Let $(M,\go,u)$ be an asymptotically Anti de Sitter solution to problem~\eqref{eq:pb_A} such that $M$ has empty boundary. If $u_{\rm min}\leq 1$, then, up to a normalization of $u$, the triple $(M,\go,u)$ is isometric to the Anti de Sitter solution~\eqref{eq:A}.
\end{corollary}

\noindent We remark that a stronger version of Corollary~\ref{cor:small_wang_A} is already known. In fact, it has been  proved by Wang in~\cite{Wang_2} that the same thesis holds for spin manifold without the hypotesis $\umin\leq 1$. Later, Qing in~\cite{Qing} has removed the spin assumption (see Theorem~\ref{thm:uniquenessW_A} and the discussion below). It is worth mentioning that the methods employed to obtain these uniqueness results heavily rely on (some kind of) the Positive Mass Theorem. More precisely, Wang's result relies on the Positive Mass Theorem for asymptotically hyperbolic manifolds, proved in~\cite{Wang_1}, whereas Qing's result exploit the Positive Mass Theorem for asymptotically flat manifolds, proved by Schoen-Yau~\cite{Sch_Yau,Sch_Yau_2}. 
%For this reason, both these approaches are hardly generalizable to the case where one has less informations on the asymptotic behavior of the triple $(M,\go,u)$.
%In other words, we can see that, the more assumptions are made on the behavior at infinity of the solution, the less effective our results are.
%
%\section{Boucher-Gibbons-Horowitz method}
%\label{sec:app_BGH}
%
%It is worth mentioning that both these approaches rely heavily on the asymptotic behavior of the triple $(M,\go,u)$, so they are hardly generalizable to the case where $(M,\go,u)$ is not asymptotically Anti de Sitter.

\section{Shen's Identity and its consequences}

\label{sec:maxprinc}

\noindent In this section we give the proofs of Theorems~\ref{thm:shen_D} and~\ref{thm:shen_A}, which consist on the analysis via the Strong Maximum Principle of Shen's Identity~\eqref{eq:div_BGH_AD}.

\subsection{Computations via Bochner formula.}

In order to prove our theorems, we need the following preparatory result, which is a simple application of the Bochner Formula.

\begin{proposition}[{Shen's Identity~\cite[formula~(8)]{Ambrozio},~\cite[formula~(12)]{Shen}}]
\label{pro:div_BGH_AD}
Let $(M,\go,u)$ be a solution of either system~\eqref{eq:pb_D} or system~\eqref{eq:pb_A}. Then it holds
\begin{equation}
\label{eq:div_BGH_AD}
{\rm div}\left[\frac{1}{u}\left(\D|\D u|^2-\frac{2}{n}\De u\D u\right)\right]\,\,=\,\,\frac{2}{u}\,\Big[\,|\DD u|^2-\frac{(\De u)^2}{n}\,\Big]\,\,\geq\,\,0\,.
\end{equation}
\end{proposition}

\begin{proof}
Since the two cases are very similar, we will do the computations for both solutions of~\eqref{eq:pb_D} and~\eqref{eq:pb_A} at the same time.
We first recall that, from the first and second equation in~\eqref{eq:pb_D} and~\eqref{eq:pb_A}, we have $\De u=\mp nu$, and $\Ric=\DD u\pm nu\go$. 
Using these equalities together with the Bochner Formula, we compute
\begin{align}
\notag
\De|\D u|^2\,&=\,2\,|\DD u|^2\,+\,2\,\Ric(\D u,\D u)\,+\,2\langle\D\De u\,|\,\D u\rangle
\\
\notag
&=\,2\,|\DD u|^2\,+\,2\,\Big[\,\frac{1}{u}\,\DD u(\D u,\D u)\,\pm\,n\,|\D u|^2\,\Big]\,\mp\,2n\,|\D u|^2
\\
\label{eq:boch_BGH_D}
&=\,2\,|\DD u|^2\,+\,\frac{1}{u}\,\langle\D|\D u|^2\,|\,\D u\rangle\,.
\end{align}
Letting
$$
Y\,=\,\D|\D u|^2\,-\,\frac{2}{n}\,\De u\,\D u\,,
$$
and using~\eqref{eq:boch_BGH_D}, we compute
\begin{align*}
{\rm div}(Y)
\,&=\, 
\De|\D u|^2\,-\,\frac{2}{n}\langle\D\De u\,|\,\D u\rangle\,-\,\frac{2}{n}(\De u)^2
\\
&=\,2\,\Big[\,|\DD u|^2-\frac{(\De u)^2}{n}\,\Big]\,+\,\frac{1}{u}\,\langle\D|\D u|^2\,|\,\D u\rangle\,\pm\,2\,|\D u|^2\,.
\end{align*}
More generally, for every nonzero $\mathscr{C}^1$ function $\alpha=\alpha(u)$, it holds
\begin{align*}
\frac{1}{\alpha}{\rm div}(\alpha Y)\,&=\,{\rm div}(Y)\,+\,\frac{\dot\alpha}{\alpha}\,\langle Y\,|\,\D u\rangle
\\
&=\,2\,\Big[\,|\DD u|^2-\frac{(\De u)^2}{n}\,\Big]\,+\,\Big(\frac{\dot\alpha}{\alpha}+\frac{1}{u}\Big)\Big(\langle\D|\D u|^2\,|\,\D u\rangle\,\pm\,2u\,|\D u|^2\Big)\,.
\end{align*}
where $\dot\alpha$ is the derivative of $\alpha$ with respect to $u$.
The computation above suggests us to choose
$$
\alpha(u)\,=\,\frac{1}{u}\,.
$$
so that $\dot\alpha/\alpha=-1/u$, and we obtain
\begin{equation}
\label{eq:div_BGH_aux_AD}
{\rm div}\Big(\,\frac{1}{u}\,Y\,\Big)\,=\,\frac{2}{u}\,\Big[\,|\DD u|^2-\frac{(\De u)^2}{n}\,\Big].
\end{equation}
The square root in the right hand side of~\eqref{eq:div_BGH_aux_AD} coincides with the $\go$-norm of the trace-free part of $\DD u$, in particular it is always positive, and the thesis follows.
\end{proof}

%{\color{blue}
\noindent Proposition~\ref{pro:div_BGH_AD} is already well known, and it has a number of applications. The most significant one is a proof of the Boucher-Gibbons-Horowitz inequality~\cite{Bou_Gib_Hor}, for which we refer the reader to the following Subsection~\ref{sub:BGH_rev}. Another interesting application of formula~\eqref{eq:div_BGH_AD} has appeared  recently in~\cite[Theorem~B]{Ambrozio}, where it is used to deduce some relevant topological features of 
the solutions to system~\eqref{eq:pb_D}.
%}

\subsection{Proof of Theorems~\ref{thm:shen_D} and~\ref{thm:shen_A}.}

In this subsection, we combine Proposition~\ref{pro:div_BGH_AD} with the Strong Maximum Principle, in order to recover Theorems~\ref{thm:shen_D} and~\ref{thm:shen_A}.
Despite the two proofs present some analogies, we prefer to prove each theorem independently. 
%{\color{blue}
We start with Theorem~\ref{thm:shen_D}, that we rewrite here in an alternative -- but equivalent -- form, for the reader's convenience.
%}

\begin{theorem}
\label{thm:shen_D_dritto}
Let $(M,\go,u)$ be a solution to problem~\eqref{eq:pb_D}, and suppose that 
\begin{equation*}
\frac{|\D u|}{\umax} \,\, \leq \,\, 1  \qquad \hbox{on} \quad \pa M \, .
\end{equation*} 
Then, up to a normalization of $u$, the triple $(M,\go,u)$ is isometric to the de Sitter solution~\eqref{eq:D}. In particular, $\pa M$ is connected.
\end{theorem}

\begin{proof}

%{\color{blue}
Combining the equation $\De u=-nu$ with formula~\eqref{eq:div_BGH_AD} in Proposition~\ref{pro:div_BGH_AD}, we get \begin{equation}
\label{eq:div_BGH_rev_D}
0 \, \leq \,
{2}\,\Big[\,|\DD u|^2-\frac{(\De u)^2}{n}\,\Big]
%\, \leq \, u \,\, {\rm div}\Big[\,\frac{1}{u}\,\D\big(|\D u|^2\,+\,u^2)\,\Big] 
\, = \, \Delta \big(|\D u|^2\,+\,u^2) \, - \, \frac{1}{u} \big\langle \,\D u \, \big| \, \D\big(|\D u|^2\,+\,u^2\big)\,\big\rangle \, .
%{\rm div}\Big[\,\frac{1}{u}\,\D\big(|\D u|^2\,+\,u^2)\,\Big]\,\geq\, 0\,.
\end{equation}
We claim that $\big(|\D u|^2\,+\,u^2)$ is constant and its value coincides with $\umax^2$. This follows essentially from the Maximum Principle, however some attention should be payed to the coefficient $1/u$, since it blows up at $\pa M$. Hence, for the sake of completeness, we prefer to present the details.
% where we have to pay some attention to the coefficient $1/u$, which is not bounded near $\pa M$.
 
As it is pointed out in Subsection~\ref{sub:prelim},
the function $u$ is analytic, and thus its critical level sets as well as its critical value are discrete.
On the other hand, one has that $|\D u| >0$ on $\pa M$, so that the zero level set of $u$ is regular. Moreover, it is possible to choose a positive number $\eta >0$ such that each level set $\{ u = \ep\}$ is regular (and diffeomorphic to $\pa M$), provided $0 < \ep \leq \eta$. Setting $M_\ep = \{u \geq \ep\}$, it is immediate to observe that the coefficient $1/u$ is now bounded above by $1/\ep$ in $M_\ep$, moreover we have that 
\begin{equation*}
\max_{M_\ep} (|\D u|^2+u^2)\, \, \leq \,\, \max_{\pa M_\ep} (|\D u|^2+u^2) \, ,
\end{equation*}
by the Maximum Principle. In particular, for every $0 < \ep \leq \eta$ it holds
\begin{equation*}
 \max_{\pa M_\eta} (|\D u|^2+u^2)\, \, \leq \,\, \max_{\pa M_\ep} (|\D u|^2+u^2)\, .
\end{equation*}
Moreover, it is easily seen that $\lim_{\ep \to 0^+} \max_{\pa M_\ep}\, (|\D u|^2+u^2) = |\D u|_{|\pa M}$, so that, using the assumption $|\D u| \leq \umax$ on $\pa M$, one gets
\begin{equation*}
 \max_{\pa M_\eta} (|\D u|^2+u^2) \, \leq \, \umax^2\, .
\end{equation*}
On the other hand, it is clear that ${\rm MAX}(u)=\{p\in M \, : \, u(p)=\umax\} \subset M_\eta$ and that for every $p \in {\rm MAX}(u)$ it holds
$(|\D u|^2+u^2) (p) \, = \, \umax^2$.
The Strong Maximum Principle implies that $|\D u|^2\,+\,u^2\equiv \umax^2$ on $M_\eta$. Since $\eta>0$ can be chosen arbitrarily small, we conclude that $|\D u|^2+u^2\equiv \umax^2$ on $M$.
%hence we can choose $\eta\in\R$ small enough so that the level set $\{u=\eta \}$ is smooth.
%Let $M_\eta=M\cap \{u\geq\eta\}$. 
%Notice that the coefficient $1/u$ is bounded on $M_\eta$ by $1/\eta$. 
%Now take $\ep<\eta$, and let $M_\ep=M\cap \{u\geq\ep\}$. From the Maximum Principle we have 
%$$
%(|\D u|^2+u^2)_{|_{\pa M_\eta}}\,\leq\,\,\max_{\pa M_\ep}\, (|\D u|^2+u^2)\,.
%$$
%On the other hand, by hypotesis, we have $|\D u|^2+u^2=|\D u|^2\leq \umax^2$ on $\pa M$, hence from the continuity of $|\D u|^2+u^2$ and the compactness of $M$, we deduce $\lim_{\ep\to 0^+}\max_{\pa M_\ep}(|\D u|^2+u^2)\leq \umax^2$, thus 
%$$
%(|\D u|^2+u^2)_{|_{\pa M_\eta}}\,\leq\,\umax^2\,.
%$$
%On the other hand, the function $|\D u|^2+u^2$ is equal to $\umax^2$ at the maximum points. Therefore, we can apply the Strong Maximum Principle in $M_\eta$ to deduce that $|\D u|^2\,+\,u^2\equiv \umax^2$ on $M_\eta$. Since $\eta$ is arbitrary, we deduce $|\D u|^2+u^2\equiv \umax^2$ on $M$.

Plugging the latter identity in formula~\eqref{eq:div_BGH_rev_D}, we easily obtain $|\DD u|^2=(\De u)^2/n$, from which it follows $\DD u=-ug$ and in turns that $\Ric=(n-1)g$, where in the last step we have used the first equation of system~\eqref{eq:pb_D}. Now we can conclude by exploiting the results in~\cite{Obata}. To this end, we double the manifold along the totally geodesic boundary, obtaining a closed compact Einstein manifold $(\hat M,\hat g)$ with $\Ric_{\hat g}=(n-1)\hat g$. On $\hat M$ we define the function $\hat u$ as $\hat u=u$ on one copy of $M$ and as $\hat u=-u$ on the other copy. Since $\DD u=0$ on $\pa M$, after the gluing the function $\hat u$ is easily seen to be $\mathscr{C}^2$ on $\hat M$. Moreover, $\hat u$ is an eigenvalue of the laplacian, and more precisely it holds $-\Delta_{\hat\go} \hat u=n\hat u$. Therefore~\cite[Theorem~2]{Obata} applies and we conclude that $(\hat M,\hat g)$ is isometric to a standard sphere.
%{\color{red} UNA PERPLESSITA': CREDO - ANZI SONO ABBASTANZA CERTO - CHE OBATA USI LE SUE IPOTESI PER DEDURRE L'EQUAZIONE DI PROPORZIONALITA' FRA HESSIANO E METRICA... CIOE' QUELLA DA CUI SIAMO PARTITI PER DIMOSTRARE CHE VALGONO LE IPOTESI DI OBATA. VABBEH...
%
%Beh s\`i, mi avevi detto di citare il teorema 2 perch\`e era quello pi\`u famoso, ma effettivamente il teorema 2 lo trova come conseguenza del teorema A (quello che ha come ipotesi la proporzionalit\`a di hessiano e metrica). Decidi tu quale dei due teoremi richiamare, alla fine \`e indifferente
%}
%}
\end{proof}

%{\color{blue}
\noindent We pass now to the proof of Theorem~\ref{thm:shen_A}, that we restate here in an alternative form. Albeit its strict analogy with the above argument, we will present the proof of Theorem~\ref{thm:shen_A} in full details, since this will give us the opportunity to show how the required adjustments are essentially related to the different topology of the manifold $M$.
%}
%
%, even if it is analogue to that of Theorem~\ref{thm:shen_D}, since some details need to be changed due to the

\begin{theorem}
\label{thm:shen_A_dritto}
Let $(M,\go,u)$ be a solution to problem~\eqref{eq:pb_A}, and suppose
$$
\liminf_{x\to\infty}\left(u^2-\umin^2-|\D u|^2\right)(x)\,\,\geq\,\, 0\,.
$$
Then, up to a normalization of $u$, the triple $(M,\go,u)$ is isometric to the Anti de Sitter solution~\eqref{eq:A}. In particular, $M$ has a unique end.
\end{theorem}

\begin{proof}
Recalling $\De u=nu$ and formula~\eqref{eq:div_BGH_AD} in Proposition~\ref{pro:div_BGH_AD}, we obtain
%{\color{blue}
\begin{equation}
\label{eq:elliptic_A}
0\,\leq\,2\Big[\,|\DD u|^2-\frac{(\De u)^2}{n}\,\Big]\,=\,\De\big(|\D u|^2\,-\,u^2\big)\,-\,\frac{1}{u}\big\langle\,\D u\,\big|\,\D\big(|\D u|^2\,-\,u^2\big)\,\big\rangle\,.
\end{equation}
We want to proceed in the same spirit as in the proof of Theorem~\ref{thm:shen_D_dritto} above. In this case the boundary is empty and the quantity $1/u$ is bounded from above by $1/\umin$ on the whole $M$. On the other hand, this time the manifold $M$ is complete and noncompact, so we have to pay some attention to the behavior of our solution along the ends.
%}
Let then $\{K_i\}_{i\in\N}$ be an exhaustion by compact sets of $M$. Without loss of generality, we can assume that the exhaustion is ordered by inclusion, namely $K_i\subset K_j$, whenever $i<j$. Applying the Weak Maxiimum Principle to the differential inequality~\eqref{eq:elliptic_A} one gets 
\begin{equation}
\label{eq:maxpr_aux_A}
\left(|\D u|^2-u^2+\umin^2 \right) (x) \,\, \leq \,\, \max_{\pa K_{i}}\left(|\D u|^2-u^2+\umin^2\right) \, ,
\end{equation}
for every $x \in K_i$ and every $i \in \mathbb{N}$. On the other hand, the assumption is clearly equivalent to $\limsup_{x \to  \infty}(|\D u|^2-u^2+\umin^2)(x)\leq 0$. This implies that, for any given $\ep>0$, there exists a large enough $j_\ep\in\N$ so that 
\begin{equation}
%\label{eq:boundcond_aux_A}
\max_{\pa K_j}\left(|\D u|^2-u^2+\umin^2\right)
%_{|_{\pa K_i}}
 \, \leq \, \ep\,,\qquad \hbox{for every $j\geq j_\ep$}\,. 
 %\ j\in\N\,.
\end{equation}
%\begin{comm}LO SO CHE QUESTO NON TI PIACER\`A, PERCH\'E STO USANDO L'IPOTESI NEL VERSO CONTRARIO. PER\`O DA UNA PARTE L'IPOTESI \`E PI\`U CARINA NEL VERSO $\liminf(u^2-\umin^2-|\D u|^2)\leq 0$ PERCH\`E SCRITTA COS\`I ASSOMIGLIA ALLA QUANTIT\`A $u^2-1-|\D u|^2$ CHE USANO QING, WANG ETC. DALL'ALTRA PARTE NELLA DIMOSTRAZIONE L'IPOTESI \`E PI\`U COMODA NEL VERSO CONTRARIO ($\limsup(|\D u|^2-u^2+\umin^2)\leq 0$), PERCH\'E QUESTO \`E IL VERSO IN CUI \`E SCRITTA L'IDENTIT\`A DI SHEN (E NON MI PIACEVA MOLTO NEMMENO L'IDEA DI SCRIVERE L'IDENTIT\`A DI SHEN AL CONTRARIO). IO LASCEREI TUTTO COS\`I. SI CAPISCE? IL TUO SENSO ESTETICO NE ESCE AMMACCATO?\end{comm}
%
%In  fact, if this were not the case, then we could find a subsequence $i_j\in\N$ and points $p_{i_j}\in \pa K_{i_j}$ such that  $(|\D u|^2-u^2+\umin^2)(p_{i_j}) \, \geq \, \ep$, and the inferior limit of this sequence would be greater than $\ep$, against our hypotesis. 
%
%Fixed $\ep>0$, and taken $i>i_\ep$, from inequality~\eqref{eq:elliptic_A} and the Weak Maximum Principle, we deduce that
%\begin{equation}
%\label{eq:maxpr_aux_A}
%|\D u|^2-u^2+\umin^2\leq\max_{\pa K_{i}}\left(|\D u|^2-u^2+\umin^2\right)<\ep
%\end{equation}
%pointwise on $K_{i_\ep}$, where in the last inequality we have used~\eqref{eq:boundcond_aux_A}. Since the same reasoning can be repeated for every $\ep>0$, from~\eqref{eq:maxpr_aux_A} we obtain
Combining the last two inequalities, we easily conclude that
$$
|\D u|^2-u^2+\umin^2\,\,\leq\,\,0\,,
$$
on the whole $M$. In particular, as soon as a compact subset $K$ of $M$ contains ${\rm MIN}(u)\,=\,\{p\in M\,:\,u(p)=\umin\}$ in its interior, we have that $\max_{\pa K}(|\D u|^2-u^2+\umin^2)\leq 0$. Since on ${\rm MIN}(u)$ it clearly holds $|\D u|^2-u^2+\umin^2=0$, the Strong Maximum Principle implies that $|\D u|^2-u^2+\umin^2\equiv 0$ on $K$. From the analyticity of $u$, it follows that $|\D u|^2-u^2+\umin^2\equiv 0$ on the whole $M$.  Plugging this information in~\eqref{eq:elliptic_A}, we easily obtain $|\DD u|^2=(\De u)^2/n$, from which we deduce $\DD u=u\go$ and we can conclude using~\cite[Lemma~3.3]{Qing}.
\end{proof}

\subsection{Boucher-Gibbons-Horowitz method revisited.}
\label{sub:BGH_rev}

In this subsection we illustrate another consequence of Proposition~\ref{pro:div_BGH_AD}, namely, we prove a local version of the well known Boucher-Gibbons-Horowitz inequality. To do that we are going to retrace the approach used in~\cite[Section~6]{Chrusciel_2}, which essentially consists in integrating identity~\eqref{eq:div_BGH_AD} on $M$ and using the Divergence Theorem. The main difference is that instead of working on the whole $M$, we will focus on a single connected component $N$ of $M\setminus{\rm MAX}(u)$. This will lead us to the proof of Theorem~\ref{thm:BGH_rev}, which we have restated here for the reader's convenience.

\begin{theorem}
Let $(M,\go,u)$ be a solution to problem~\eqref{eq:pb_D}, let $N$ be a 
connected component of $M\setminus{\rm MAX}(u)$, and let $\pa N = \pa M \cap N$ be the non-empty and possibly disconnected boundary portion of $\pa M$ that lies in $N$. Suppose also that $|\D u|^2\leq C(\umax-u)$ on the whole $N$, for some constant $C\in\R$. Then it holds
$$
0 \,\,\, \leq \, \int_{\pa N}\!\!|\D u|\left[\RRR^{\pa N}-(n-1)(n-2)\right]\rmd\sigma\,,
$$
%\int_{\pa N}|\D u|\left[\RRR^{\pa N}-(n-1)(n-2)\right]\rmd\sigma\,\geq\,0\,,
%$$
where $\RRR^{\pa N}$ is the scalar curvature of the restriction of the metric $\go$ to $\pa N$. Moreover, if the equality holds then, up to a normalization of $u$, $(M,\go,u)$ is isometric to the de Sitter solution~\eqref{eq:S}.
\end{theorem}

\begin{proof}
In Subsection~\ref{sub:prelim} we have recalled that $|\D u|\neq 0$ on $\pa M=\{u=0\}$, and that the critical values of $u$ are always discrete. Therefore, from the compactness of $M$ and the properness of $u$, it follows that we can choose $\ep>0$ so that the level sets $\{u=t\}$ are regular for all $0\leq t\leq\ep$ and for all $\umax-\ep\leq t<\umax$. From Proposition~\ref{pro:div_BGH_AD} we have
\begin{equation*}
%\label{eq:div_BGH_AD}
{\rm div}\left[\frac{1}{u}\left(|\D u|^2+u^2\right)\right]\,\,=\,\,\frac{2}{u}\,\Big[\,|\DD u|^2-\frac{(\De u)^2}{n}\,\Big]\,\,\geq\,\,0\,.
\end{equation*}
To simplify the computations, we are going to integrate by parts the inequality
\begin{equation}
\label{eq:stolta}
{\rm div}\left[\frac{1}{u}\left(|\D u|^2+u^2\right)\right]\,\,\geq \,\,0 \, .
\end{equation}
Proceeding in this way, we are going to prove the validity of the inequality mentioned in the statement of the theorem. In order to deduce the rigidity one has to keep into account also the quadratic term 
$$\Big[\,|\DD u|^2-\frac{(\De u)^2}{n}\,\Big] \, .$$
However, since this part of the argument is completly similar to what we have done in the previous subsection, we omit the details, leaving them to the interested reader. Integrating inequality~\eqref{eq:stolta} on the domain $\{\ep<u<\umax-\ep\}\cap N$ and using the Divergence Theorem, we obtain
\begin{equation}
\label{eq:int_part_fin_BGH_D}
\int_{\{u=\umax-\ep\}\cap N} \left\langle \frac{\D\big(|\D u|^2\!+u^2)}{u}\,\Bigg|\, \nu\right\rangle\,\, \rmd\sigma\,\,\,\geq \,\,\,\int_{\{u=\ep\}\cap N} \left\langle \frac{\D\big(|\D u|^2\!+u^2)}{u}\,\Bigg|\, \nu\right\rangle\,\, \rmd\sigma \, ,
%\frac{1}{u}\, \Big\langle \D\big(|\D u|^2\!+u^2)\,\Big|\, \nu\Big\rangle\, \rmd\sigma \,,
\end{equation}  
where we have used the short hand notation $\nu=\D u/|\D u|$, for the unit normal to the set $\{\ep<u<\umax-\ep\}\cap N$. 
%On the other hand, it is easy to show that it holds
%$$
%\HHH\,=\,\frac{\De u}{|\D u|^2}-\frac{\DD u(\D u,\D u)}{|\D u|^3}\,,
%$$
Using the first equation in~\eqref{eq:pb_D}, we get
\begin{align*}
\left\langle \frac{\D\big(|\D u|^2\!+u^2)}{u}\,\Bigg|\, \nu\right\rangle
%\frac{1}{u}\Big\langle \D\big(|\D u|^2\,+\,u^2)\,\Big|\, \nu\Big\rangle
\, &= \, {2}\left[\,  \frac{\DD u(\D u,\D u)+u|\D u|^2}{{u\,|\D u|}} \, \right]
\\
&= \, {2}\left[ \, \frac{\Ric(\D u,\D u)\,-\,n\,|\D u|^2+\,|\D u|^2}{{|\D u|}} \, \right] \, = \, 2\,|\D u|\left[ \, \Ric(\nu,\nu)\,-\,(n-1) \, \right]\,.
\end{align*}
In view of this identity, inequality~\eqref{eq:int_part_fin_BGH_D} becomes
\begin{equation}
\label{eq:div_aux}
\int_{\{u=\umax-\ep\}\cap N}\!\!\!\!\!|\D u| \, \left[\,  \Ric(\nu,\nu)-(n-1)\, \right] \,\,\rmd\sigma\,\,\,\geq \,\,\, \int_{\{u=\ep\}\cap N}\!\!\!\!\!|\D u|\, \left[ \Ric(\nu,\nu)-(n-1)\right]\,\, \rmd\sigma\,.
\end{equation}
We now claim that the $\liminf$ of the left hand side vanishes when $\ep \to 0$.
%We want to show that there exists a sequence of values of the parameter $\ep$, converging to zero, so that the second integral in~\eqref{eq:div_aux} goes to zero.
Since $M$ is compact and $\go$ is smooth, the quantity $\Ric(\nu,\nu)-(n-1)$ is continuous, thus bounded, on $M$. Therefore, to prove the claim, it is sufficient to show that 
\begin{equation*}
%\label{eq:liminf_aux}
\liminf_{t \to \umax}\int_{\{u=t\}\cap N}\!\!\!|\D u|\,\,\rmd\sigma\,=\,0\,.
\end{equation*}
If this is not the case, then we can suppose that the $\liminf$ in the above formula is equal to some positive constant $\delta>0$. This means that up to choose a small enough $\alpha>0$, we could insure that 
\begin{equation*}
\int_{\{u=t\}\cap N}\!\!\!|\D u|\,\,\rmd\sigma \,\, \geq \,\, \frac{\delta}{2} \, , \quad \hbox{for $\umax - \alpha < t <\umax$} \,.
\end{equation*}
Combining this fact with the coarea formula, one has that for every $0< \ep < \alpha$, it holds
\begin{align*}
\int_{\{\umax-\alpha<u<\umax - \ep \}\cap N}\left(\frac{|\D u|^2}{\umax-u}\right)\,\,\rmd\mu &\,\,=\,\,\int_{\umax-\alpha}^{\umax - \ep}  \!\!\frac{\rmd t}{\umax -t}\int_{\{u=t\}\cap N}\!\!\!{|\D u|}\,\rmd\sigma  \,\, \geq \\
\geq \,\,  \frac{\delta}{2} \, \int_{\umax-\alpha}^{\umax - \ep}  \!\!\frac{\rmd t}{\umax -t} &\,\, = \,\, \frac{\delta}{2} \, \int_{\ep}^{\alpha}  \frac{\rmd \tau}{\tau} \,  .
\end{align*}
Now, in view of the assumption $|\D u|^2\leq C(\umax-u)$, the leftmost hand side is bounded above by $C \, | N|$. On the other hand, the rightmost hand side tends to $+\infty$, as $\ep \to 0$. Since we have reached a contradiction, the  claim is proven. Hence, taking the $\liminf_{\ep \to 0}$ in~\eqref{eq:div_aux}, we arrive at
\begin{equation}
\label{eq:BGH_aux}
\int_{\pa N}|\D u|\left[ -\,\Ric(\nu,\nu)\,+\,(n-1)\right]\, \rmd\sigma\,\geq\,0\,.
\end{equation}
To conclude, we observe that, on the totally geodesic boundary $\pa N$ of our connected component, the Gauss equation reads
$$
-\Ric(\nu,\nu)\,=\,\frac{\RRR^{\pa N}-\RRR}{2}\,=\,
%\frac{1}{2}\,\big[\,
\frac{\RRR^{\pa N}-\,n(n-1)}{2} \, .
%\,\big]\,.
$$
Substituting the latter identity in formula~\eqref{eq:BGH_aux} we obtain the thesis.
\end{proof}

\section{Local lower bound for the surface gravity}
\label{sec:monoform}

\noindent In this section we focus on the case $\Lambda>0$, and we are going to present the
complete proof of Theorem~\ref{thm:main_D}. As discussed in Subsection~\ref{sub:surfmass}, the local nature of this lower bound for the {\em surface gravity} is at the basis of our definition of {\em virtual mass}, as explained in Theorem~\ref{thm:PMS}.

\subsection{Some preliminary results}

%Preliminaries: other consequences of the (Weak) Maximum Principle
\label{sub:monoform}

%We discuss now some preparatory results, which are simple applications of the Weak Maximum Principle.
As usual, we denote by $N$ a connected component of $M \setminus {\rm MAX}(u)$. The next lemma shows that the set $\pa N = \pa M \cap N$ is always nonempty, and thus it is necessarily given by a disjoint union of {\em horizons}. 
%We will denote by $\overline{N}$ the closure of $N$ in $M$. The domain $\overline{N}$ is compact, but its boundary is not necessarily smooth. In fact, in general we have $\pa\overline{N}\cap{\rm MAX}(u)\neq\emptyset$ and we have already observed that ${\rm MAX}(u)$ is not necessarily a smooth submanifold. It is also important to notice that $\pa\overline{N}$ cannot be completely contained in ${\rm MAX}(u)$, that is, the set $\pa N=\pa\overline{N}\cap\pa M=N\cap\pa M$ is nonempty, as it is shown in the following lemma.

\begin{lemma}[No Islands Lemma, $\Lambda>0$]
\label{lem:noisole}
Let $(M,\go,u)$ be a solution to problem~\eqref{eq:pb_D} and let $N$ be a 
%(not necessarily unique) 
connected component of $M\setminus{\rm MAX}(u)$. Then $N \cap \pa M \neq \emptyset$.
\end{lemma}
\begin{proof} 
Let $N$ be a connected component of $M \setminus {\rm MAX}(u)$ and assume by contradiction that $N \cap \pa M = \emptyset$. Since ${\rm MAX}(u) \cap \pa M = \emptyset$, one has that $\overline{N} \setminus N \subseteq {\rm MAX}(u)$, where we have denoted by $\overline{N}$ the closure of $N$ in $M$. On the other hand, the scalar equation in~\eqref{eq:pb_D} implies that $\Delta u \leq 0$ in $N$ and thus, by the Weak Minimum Principle , one can deduce that
\begin{equation*}
\min_{\overline N} u \,\, = \,\, \min_{\overline{N} \setminus N} u \,\, \geq \,\, \min_{{\rm MAX}(u)} u \,\, = \,\, \umax \, .
\end{equation*}
In other words $u \equiv \umax$ on $N$. Since $N$ has non-empty interior, $u$ must be constant on the whole $M$, by analyticity. This yields the desired contradiction.
\end{proof}

\noindent As an easy application of the Maximum Principle, we obtain the following gradient estimate, which is the first step in the proof of the main result.

\begin{lemma}
\label{le:shen_aux_D}
Let $(M,\go,u)$ be a solution to problem~\eqref{eq:pb_D}, let $N$ be a 
connected component of $M\setminus{\rm MAX}(u)$, and let $\pa N = \pa M \cap N$ be the non-empty and possibly disconnected boundary portion of $\pa M$ that lies in $N$.
If $|\D u|\leq \umax$ on $\pa N$, then it holds $|\D u|^2\leq \umax^2-u^2$ on the whole $N$.
\end{lemma}

\begin{proof}
The thesis will essentially follow by the Maximum Principle applied to the equation~\eqref{eq:div_BGH_rev_D} on the whole domain $N$. However, as in the proof of Theorem~\ref{thm:shen_A_dritto}, we have to pay attention to the coefficient $1/u$, which blows up at the boundary $\pa N$. We also notice that in general the set $\overline{N} \cap {\rm MAX}(u)$ is not necessarily a regular hypersurface. Albeit this does not represent a serious issue for the applicability of the Maximum Principle, we are going to adopt the same treatment for both $\pa N$ and $\overline{N} \cap {\rm MAX}(u)$, considering  subdomains of the form $N_\ep = N \cap \{\ep\leq u\leq \umax-\ep\}$, for $\ep$ sufficiently small. To be more precise, we first recall from Subsection~\ref{sub:prelim} that the function $u$ is analytic and then the set of its critical values is discrete. Therefore, there exists $\eta>0$ such that, for every $0<\ep\leq\eta$ the level sets $\{u=\ep \}$ and $\{u=\umax-\ep\}$ are regular. Applying the Maximum Principle to equation~\eqref{eq:div_BGH_rev_D}, we get
\begin{equation*}
\max_{N_\ep} (|\D u|^2+u^2) \, \leq \, \max_{\pa N_\ep} (|\D u|^2+u^2) \,.
\end{equation*}
On the other hand we have that $|\D u|^2+u^2=\umax^2$ on ${\rm MAX}(u)$, and $|\D u|^2+u^2\leq \umax^2$ on $\pa N$, by our assumption. Hence, letting $\ep \to 0$ in the above inequality, we get the desired conclusion.
%
%
%\bigskip
%\bigskip
%
%We have $|\D u|^2+u^2=\umax^2$ on ${\rm MAX}(u)$ and from the hypotesis $|\D u|^2+u^2\leq \umax^2$ on $\pa N$. Since $|\D u|^2+u^2$ is smooth on $N$ and $M$ is compact, denoting by $N_\ep$ the compact set $\{\ep\leq u\leq \umax-\ep\}\cap N$, we have 
%$$
%\lim_{\ep\to 0^+}\max_{\pa N_\ep}(|\D u|^2+u^2)\,\leq\, \umax^2\,.
%$$}
%Since the quantity $|\D u|^2+u^2$ satisfies the elliptic inequality~\eqref{eq:div_BGH_rev_D}, applying the Weak Maximum Principle on $N_\ep$ for an arbitrarily small $\ep>0$, we obtain the thesis.
\end{proof}

\subsection{Monotonicity formula.}
\label{sub:monotonicity_D}

Let $(M,\go,u)$ be a solution to problem~\eqref{eq:pb_D}, and let $N$ be a connected component of $M\setminus{\rm MAX}(u)$. Proceeding in analogy with~\cite{Ago_Maz_2,Bor_Maz}, we introduce the function $U:[0,\umax)\rightarrow\R$ given by
\begin{equation}
\label{eq:Up_D}
t \,\, \longmapsto \,\, U(t) \, = \, \Big(\frac{1}{\umax^2-t^2}\Big)^{\!\!\frac{n}{2}}\!\!\!\!\!\! \int\limits_{ \{ u = t \}\cap N} \!\!\!\!  |\D u| \, \rmd \sigma \,.
\end{equation}
We remark that the function $t \mapsto U(t)$ is well defined, since the integrand is globally bounded and the level sets of $u$ have finite hypersurface area. In fact, since $u$ is analytic (see~\cite{Chrusciel_1,ZumHagen}), the level sets of $u$ have locally finite $\mathscr{H}^{n-1}$--measure by the results in~\cite{Kra_Par}. Moreover, they are compact and thus their hypersurface area is finite. To give further insights about the definition of the function $t \mapsto U(t)$, we note that, using the explicit formul\ae~\eqref{eq:D}, one easily realizes that the quantities
\begin{equation}
\label{eq:Pfunct_confvol_D}
M \, \ni \,x \, \longmapsto \,\,  \frac{|\D u|}{\sqrt{\umax^2-u^2}}(x) \quad \hbox{and} \quad  [0, \umax) \, \ni \, t \, \longmapsto \!\!\!\int\limits_{\{u = t\}}\!\!
 \Big(\frac{1}{\umax^2-u^2}\Big)^{\!\frac{n-1}{2}}
 \rmd\sigma 
\end{equation}
are constant on the de Sitter solution. We notice that the function $t \mapsto U(t)$ can be rewritten in terms of the above quantities as
\begin{equation*}
{ U}(t)\,\,=\!\!\!
\int\limits_{\{u = t\}\cap N}\!\!\!
\left( \frac{|\D u|}{\sqrt{\umax^2-u^2}} \right) \Big(\frac{1}{\umax^2-u^2}\Big)^{\!\frac{n-1}{2}}
\,\, \rmd\sigma \,,
\end{equation*}
hence the function $t \mapsto U(t)$ is constant on the de Sitter solution. In the next proposition we are going to show that, for a general solution, the function $U$ is monotonically nonincreasing, provided the surface gravity of the connected component of $\pa N$ is bounded above by $1$.
\begin{proposition}[Monotonicity, case $\Lambda>0$]
\label{pro:main_D}
Let $(M,\go,u)$ be a solution to problem~\eqref{eq:pb_D}, let $N$ be a 
connected component of $M\setminus{\rm MAX}(u)$, and let $\pa N = \pa M \cap N$ be the non-empty and possibly disconnected boundary portion of $\pa M$ that lies in $N$. If $|\D u|\leq \umax$ on $\pa  N$, then the function $U(t)$ defined in~\eqref{eq:Up_D} is monotonically nonincreasing. 
\end{proposition}

\begin{proof}
Recalling $\De u=-nu$, we easily compute
\begin{align}
\notag
{\rm div}\left[\frac{\D u}{(\umax^2-u^2)^{\frac{n}{2}}}\right]\,&=\,\frac{\De u}{(\umax^2-u^2)^{\frac{n}{2}}}+n\,u\,\frac{|\D u|^2}{(\umax^2-u^2)^{\frac{n}{2}+1}}
\\
\label{eq:div_Du_D}
\,&=\,-\frac{n\,u}{(\umax^2-u^2)^{\frac{n}{2}+1}}\left(\umax^2-u^2-|\D u|^2\right)\,\leq\,0\,,
\end{align}
where the last inequality follows from Lemma~\ref{le:shen_aux_D}. Integrating by parts inequality~\eqref{eq:div_Du_D} in $\{t_1\leq u\leq t_2\}\cap N$ for some $t_1<t_2$, and applying the Divergence Theorem, we deduce
\begin{multline}
\label{eq:int_div_Du_S}
\int\limits_{\{u=t_1\}\cap N}\!\!\!\!\bigg\langle\frac{\D u}{(\umax^2-u^2)^{\frac{n}{2}}}\,\bigg|\,{\rm n}\bigg\rangle\,\rmd\sigma
\,+
\!\!\!\int\limits_{\{u=t_2\}\cap N}\!\!\!\!\bigg\langle\frac{\D u}{(\umax^2-u^2)^{\frac{n}{2}}}\,\bigg|\,{\rm n}\bigg\rangle\,\rmd\sigma
\,=\,
\\
-\!\!\!\int\limits_{\{t_1\leq u\leq  t_2\}\cap N}\frac{n\,u}{(\umax^2-u^2)^{\frac{n}{2}+1}}\left(\umax^2-u^2-|\D u|^2\right)\,\leq\,0\,,
\end{multline}
where ${\rm n}$ is the outer $\go$-unit normal to the boundary of the set $\{t_1\leq u\leq t_2\}$. In particular, one has ${\rm n}=-\D u/|\D u|$ on $\{u=t_1\}$ and ${\rm n}=\D u/|\D u|$ on $\{u=t_2\}$, thus formula~\eqref{eq:int_div_Du_S} rewrites as
$$
\int\limits_{\{u=t_2\}\cap N}\!\!\frac{|\D u|}{(\umax^2-u^2)^{\frac{n}{2}}}\,\rmd\sigma\,\,\,\leq \!\!\!\int\limits_{\{u=t_1\}\cap N}\!\!\frac{|\D u|}{(\umax^2-u^2)^{\frac{n}{2}}}\,\rmd\sigma
\,,
$$
from which it follows $U(t_2)\leq U(t_1)$, as wished.
\end{proof}

\subsection{Proof of Theorem~\ref{thm:main_D}.}
\label{sub:main_proof_D}

In the previous subsection, we have shown the monotonicity of the function $U$. In order to prove Theorem~\ref{thm:main_D}, we also need an estimate of the behavior of $U(t)$ as $t$ approaches $\umax$. This is the content of the following proposition.

\begin{proposition}
\label{pro:estimate_D}
Let $(M,\go,u)$ be a solution to problem~\eqref{eq:pb_D}.
Let $N$ be a connected component of $M\setminus{\rm MAX}(u)$ and let $U$ be the function defined as in~\eqref{eq:Up_D}.
If $\mathscr{H}^{n-1}\big({\rm MAX}(u)\cap \overline{N}\big)> 0$, then $\lim_{t\to \umax^-}U(t)=+\infty$. 
\end{proposition}

\begin{proof}
From the {\L}ojasiewicz inequality (see~\cite[Th\'{e}or\`{e}me~4]{Lojasiewicz_1} or~\cite{Kur_Par}), we know that for every point $p\in{\rm MAX}(u)$ there exists a neighborhood $p\ni V_p\subset M$ and real numbers $c_p>0$ and $0<\theta_p<1$, such that for each $x\in V_p$ it holds
$$
|\D u|(x)\,\geq\,c_p\,[\umax-u(x)]^{\theta_p}\,.
$$
Up to possibly restricting the neighborhood $V_p$, we can suppose $\umax-u<1$ on $V_p$, so that for every $x\in V_p$ it holds
$$
|\D u|(x)\,\geq\,c_p\,[\umax-u(x)]\,.
$$ 
Since ${\rm MAX}(u)$ is compact, it is covered by a finite number of sets $V_{p_1},\dots,V_{p_k}$. In particular, setting $c=\min\{c_{p_1},\dots,c_{p_k}\}$, the set $V=V_{p_1}\cup\dots\cup V_{p_k}$ is a neighborhood of ${\rm MAX}(u)$ and the inequality
\begin{equation}
\label{eq:loja}
|\D u|\,\geq\,c\,(\umax-u)
\end{equation}
is fulfilled on the whole $V$. Now we notice that the function $U(t)$ can be rewritten as follows
$$
U(t)
\,\,=\,\,
\bigg(\frac{1}{\umax^2-t^2}\bigg)^{\!\!\frac{n-2}{2}}\!\!\!\!\!\!\!\int\limits_{\{u=t\}\cap N}\!\!\!\frac{|\D u|}{(\umax-u)(\umax+u)} \,\,\rmd\sigma\,.
$$
Thanks to the compactness of $M$ and to the properness of $u$, it follows that for $t$ sufficiently close to $\umax$ we have $\{u=t\}\cap N\subset V$. For these values of $t$, using inequality~\eqref{eq:loja} we obtain the following estimate
$$
U(t)
\,\,\,\geq\,\,\,\frac{c}{2\umax}\,
\bigg(\frac{1}{\umax^2-t^2}\bigg)^{\!\!\frac{n-2}{2}}\!\!\!\!\cdot\,\,\big|\{u=t\}\cap N\big|\,.
$$
Therefore, in order to prove the thesis, it is sufficient to show that if $\mathscr{H}^{n-1}\big({\rm MAX}(u)\cap \overline{N}\big)> 0$, then
\begin{equation}
\label{eq:thesis_estimate_D}
\limsup_{t\to \umax^-}\big|\{u=t\}\cap N\big|>0\,.
\end{equation}
To this end, we recall that, since $u$ is analytic and $\mathscr{H}^{n-1}\big({\rm MAX}(u)\cap \overline{N}\big)> 0$, it follows from~\cite{Lojasiewicz_2} (see also~\cite[Theorem~6.3.3]{Kra_Par}) that the set ${\rm MAX}(u)\cap \overline{N}$ contains a smooth non-empty, relatively open hypersurface $\Sigma$ such that $\mathscr{H}^{n-1}\big(({\rm MAX}(u)\cap\overline{N})\setminus\Sigma\big)=0$. In particular, given a point $p$ on $\Sigma$, we are allowed to consider an open neighbourhood $\Omega$ of $p$ in $M$, where the signed distance to $\Sigma$
$$
r(x)\,=\,
\begin{cases}
+ \, d(x,\Sigma)   & \text{ if } x\in \Omega \cap N\,,
\\
- \, d(x,\Sigma)  & \text{ if } x\in \Omega\setminus N\,.
\end{cases}
$$
is a well defined smooth function (see for instance~\cite{Foote,Kra_Par_smoothdist}, where this result is discussed in full details in the Euclidean setting, however, as it is observed in~\cite[Remarks~(1) and~(2)]{Foote}, the proofs extend with small modifications to the Riemannian setting). In order to prove~\eqref{eq:thesis_estimate_D}, we are going to perform a local analysis, in a compact cylindrical neighborhood $C_\delta \subset \Omega$ of $p$. Let us define such a neighborhood and set up our framework:
\begin{itemize}
\item First consider a smooth embedding $F_0$ of the $(n-1)$-dimensional closed unit ball $\overline{B^{n-1}}$ into $M$ 
$$
F_0 \, : \, \overline{B^{n-1}} \, \hookrightarrow \, M \, , \qquad (\theta^1,\ldots, \theta^{n-1}) \, \mapsto \, F_0 (\theta^1,\ldots, \theta^{n-1}) 
$$
such that $\Sigma_0 \, = \, F_0(\overline{B^{n-1}})$ is strictly contained in the interior of  $\Sigma \cap \Omega$.
\smallskip
\item Given a small enough real number $\delta>0$, use the flow of $\D r$ to extend the map $F_0$ to the cartesian product $[-\delta, \delta] \times \overline{B^{n-1}}$, obtaining a new map 
$$
F \, : \, [-\delta, \delta] \times \overline{B^{n-1}} \, \hookrightarrow \, M \, , \qquad (\rho,\theta^1,\ldots, \theta^{n-1}) \, \mapsto \, F (\rho,\theta^1,\ldots, \theta^{n-1}) 
$$
satisfying the initial value problem
$$
\frac{dF}{d\rho} \, = \, \D r \circ F \, , \qquad F(0,\,  \cdot\,) \, = \, F_0(\,\cdot\,) \, .
$$
It is not hard to check that the relation $r\big(F(\rho, \theta^1,\ldots, \theta^{n-1})\big) = \rho$ must be satisfied, so that for every $\rho \in [-\delta, \delta]$, the image $\Sigma_\rho = F(\rho,\overline{B^{n-1}})$ belongs to the level set $\{r = \rho \}$ of the signed distance. 
%In other words, the maps $F_\rho = F(\rho, \, \cdot \,) : \overline{B^{n-1}} \, \hookrightarrow \, M $ are smooth embeddings of $\overline{B^{n-1}}$
\smallskip
\item Define the cylindrical neighbourhood $C_\delta$ of $p$ simply as $F\big( [-\delta, \delta] \times \overline{B^{n-1}} \big)$. By construction, the map $F$ is a parametrisation of $C_\delta$. Moreover, still denoting by $g$ the metric pulled-back from $M$ through the map $F$, we have that 
%this can be written in term of the parameters $\rho, \theta^1,\ldots, \theta^{n-1}$ as
\begin{equation*}
\go=d\rho\otimes d\rho\,+\,g_{ij}(\rho,\theta^1,\dots,\theta^{n-1})d\theta^i\otimes d\theta^j\,,
\end{equation*}
where the $g_{ij}$'s are smooth functions of the coordinates $(\rho, \theta^1,\ldots, \theta^{n-1})$ of $[-\delta, \delta] \times \overline{B^{n-1}}$. 
In particular, for any fixed $0<\ep<1$, we can suppose that, up to diminishing the value of $\delta>0$, the following estimates hold true
\begin{equation}
\label{eq:estimate_on_go}
(1-\ep)^2\,g_{ij}(0,\theta) \, \leq \,g_{ij}(\rho,\theta) \, \leq \, (1+\ep)^2\,g_{ij}(0,\theta)\,,
\end{equation} 
for every $\theta=(\theta^1,\dots,\theta^{n-1})\in \overline{B^{n-1}}$ and every $\rho\in[-\delta,\delta]$.
\smallskip
\item Finally, let $
u_\delta\,=\,\max_{\Sigma_\delta}u$. It follows from the construction that $u_\delta < \umax$. For $u_\delta \leq t \leq \umax$, we are going to consider the (pulled-back) level sets of $u$ given by
$$
L_t \,\, = \,\, F^{-1} \big( \{u=t \} \cap \overline{N} \big) \,\,\subset \,\, [0, \delta] \times \overline{B^{n-1}}  \, ,
$$
together with their natural projection on $L_{\umax} = \{0\}\times\overline{B^{n-1}}$. These are defined by
$$
\pi_t \, : \,L_t\,\longrightarrow \, \{0\}\times\overline{B^{n-1}}  \,,\qquad \pi_t\,: \,(\rho,\theta)\,\longmapsto\,(0,\theta) \, .
$$
It is not hard to see that for $u_\delta \leq t \leq \umax$, the projection $\pi_t$ is surjective. This follows from the fact that for any given $\theta \in \overline{B^{n-1}}$ the assignment 
$$
[0,\delta] \ni \rho \longmapsto (u \circ F) (\rho, \theta)
$$
is continuous and its range contains the closed interval $[u_\delta, \umax]$.
\end{itemize}
\begin{figure}
\centering
\includegraphics[scale=1]{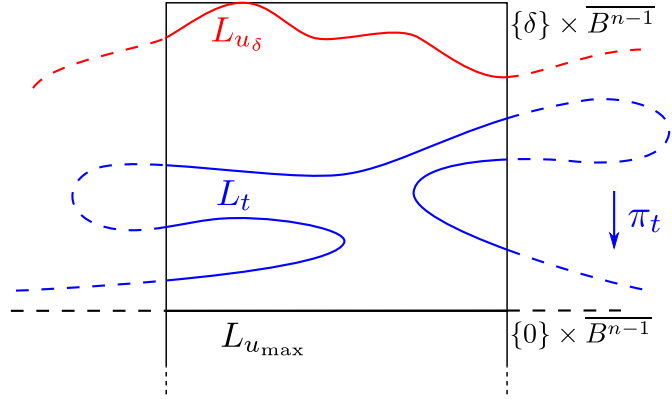}
\caption{\small
A section of the cylinder $[-\delta,\delta]\times\overline{B^{n-1}}$. The arrow shows the action of the function $\pi_t$, that sends the points of $L_t$ to their projection on $L_{\umax}=\{0\}\times\overline{B^{n-1}}$.
}
\end{figure}
With the notations introduced above, we claim that for every $u_\delta \leq t \leq \umax$ and every connected open set $S \subset L_t$, we have 
\begin{equation}
\label{eq:claim_aux}
{\rm diam}_{\go}(S)\,\geq\,({1-\ep})\,\,{\rm diam}_{\go}(\pi_t(S))\,.
\end{equation}
Since we have already shown that $\pi_t$ is surjective, it follows from the very definition of the Hausdorff measure that the claim implies the inequality
$$
\mathscr{H}^{n-1} \big(L_t\big)\,\geq\,{(1-\ep)^{n-1}}\,\,\mathscr{H}^{n-1} \big(L_{\umax} \big)\,,
$$ 
which is clearly equivalent to~\eqref{eq:thesis_estimate_D}.
To prove~\eqref{eq:claim_aux}, let us fix $t\in[u_\delta,\umax]$ and consider a $\mathscr{C}^1$ curve
$$
\gamma\,:\, I\longrightarrow L_t\,,\qquad s\longmapsto \gamma(s)\,=\,\big(\rho(s),\theta(s)\big)\,,
$$
where $I\subset\R$ is an interval.
We want to show that the lenght of $\gamma$ is controlled from below by the lenght of its projection $\pi_t\circ\gamma$, which is the curve on $L_{\umax}$ defined by $(\pi_t\circ\gamma)(s)=(0,\theta(s))$ for every $s\in I$.
Recalling the expression of $\go$ with respect to the coordinates $(\rho,\theta)$ and the estimate~\eqref{eq:estimate_on_go}, we compute
\begin{align*}
\left|\frac{d\gamma}{ds}\right|_g^2\!\!(s)\,\,&=\,\,\left|\frac{d\rho}{ds}\right|_g^2\!\!(s)\,+\,g_{ij}\big(\rho(s),\theta(s)\big)\,\frac{d\theta^i}{ds}(s)\,\frac{d\theta^j}{ds}(s)
\\
&\geq\,\,(1-\ep)^2\,g_{ij}\big(0,\theta(s)\big)\,\frac{d\theta^i}{ds}(s)\,\frac{d\theta^j}{ds}(s)
\\
&=\,\,(1-\ep)^2\,\left|\frac{d(\pi_t\circ\gamma)}{ds}\right|_g^2\!\!(s)\,.
\end{align*}
In particular, the same inequality holds between the lenghts of $\gamma$ and its projection $\pi_t\circ\gamma$. Claim~\eqref{eq:claim_aux} follows. 
\end{proof}

\noindent 
Combining Propositions~\ref{pro:main_D} and~\ref{pro:estimate_D} we easily obtain Theorem~\ref{thm:main_D}, that we restate here -- in an alternative form -- for the ease of reference.

\begin{theorem}
\label{thm:main_D_dritto}
Let $(M,\go,u)$ be a solution to problem~\eqref{eq:pb_D}, let $N$ be a 
connected component of $M\setminus{\rm MAX}(u)$, and let $\pa N = \pa M \cap N$ be the boundary portion of $\pa M$ that lies in $N$. Suppose that
\begin{equation*}
\frac{|\D u|}{\umax} \,\, \leq \,\, 1  \qquad \hbox{on} \quad \pa N \, .
\end{equation*} 
Then, up to a normalization of $u$, the triple $(M,\go,u)$ is isometric to the de Sitter solution~\eqref{eq:D}. In particular $\pa M$ and $M\setminus {\rm MAX}(u)$ are both connected.
\end{theorem}

\begin{proof}
Let us consider the function $t \mapsto U(t)$ defined in~\eqref{eq:Up_D}. Thanks to the assumption $|\D u|\leq \umax$ on $\pa N$, we have that Proposition~\ref{pro:main_D} is in force, and thus $t \mapsto U(t)$ is monotonically nonincreasing. In particular, we get	$$
	\lim_{t\to \umax^-}U(t)\,\leq\, U(0)\,=\,\int\limits_{\pa N}|\D u|\,\rmd\sigma\,\leq \umax\,|\pa N|\,<\,\infty\,.
	$$
%	where we have used the hypotesis $|\D u|\leq \umax$ on $\pa N$.
In light of Proposition~\ref{pro:estimate_D}, this fact tells us that $\mathscr{H}^{n-1}\big({\rm MAX}(u)\cap \overline{N}\big)= 0$. This means that ${\rm MAX}(u)\cap \overline{N}$ cannot disconnect the domain $N$ from the rest of the manifold $M$. In other words, $N$ is the only connected component of $M\setminus{\rm MAX}(u)$. In particular particular, $\pa M\cap N=\pa M$ and Theorem~\ref{thm:shen_D} applies, giving the thesis.
\end{proof}

\section{A characterization of the Anti de Sitter spacetime}

\noindent In this section we focus on the case $\Lambda<0$ and, proceeding in analogy with Section~\ref{sec:monoform}, we prove  Theorem~\ref{thm:main_A}.

\subsection{Some preliminary results.}
%Preliminaries: other consequences of the (weak) Maximum Principle

\noindent Here we prove the analogues of Lemmata~\ref{lem:noisole} and~\ref{le:shen_aux_D}. 
For a connected component $N$ of $M\setminus {\rm MIN}(u)$, we will denote by $\overline{N}$ the closure of $N$ in $M$. Notice that $\overline{N}$ is a manifold with boundary $\pa\overline{N}={\rm MIN}(u)\cap\overline{N}$. 
 Since ${\rm MIN}(u)$ might be singular, the boundary $\pa\overline{N}$ is not necessarily smooth in general. 
 Another important feature of $\overline{N}$ 
is that it must be noncompact, as we are going to show in the following lemma.
\begin{lemma}[No Islands Lemma, $\Lambda<0$]
\label{lem:noisole_A}
Let $(M,\go,u)$ be a solution to problem~\eqref{eq:pb_A} and let $N$ be a 
%(not necessarily unique) 
connected component of $M\setminus{\rm MIN}(u)$. Then $\overline{N}$ has at least one end.
\end{lemma}
\begin{proof} 
Let $N$ be a connected component of $M \setminus {\rm MIN}(u)$ and assume by contradiction that $\overline{N}$ has no ends. In particular, $\overline{N}$ is compact, and since also ${\rm MIN}(u)$ is compact, one has that $\overline{N} \setminus N \subseteq {\rm MIN}(u)$. On the other hand, from~\eqref{eq:pb_A} we have $\Delta u \geq 0$ in $N$, hence, by the Weak Maximum Principle, one obtains
\begin{equation*}
\max_{\overline N} u \,\, = \,\, \max_{\overline{N} \setminus N} u \,\, \leq \,\, \max_{{\rm MIN}(u)} u \,\, = \,\, \umin \, .
\end{equation*}
This implies that $u \equiv \umin$ on $N$. Since $N$ has non-empty interior, $u$ must be constant on the whole $M$, by analyticity. This yields the desired contradiction.
\end{proof}

\noindent A similar application of the Maximum Principle leads to the following result, which is the analogue of Lemma~\ref{le:shen_aux_D} in the case $\Lambda<0$.

\begin{lemma}
	\label{le:shen_aux_A}
Let $(M,\go,u)$ be a solution to problem~\eqref{eq:pb_A}, and let $N$ be a connected component of $M\setminus{\rm MIN}(u)$.  If 
$$
\liminf_{x\in N,\,x\to\infty}\left(u^2-\umin^2-|\D u|^2\right)\geq 0\,,
$$
then it holds $|\D u|^2\leq u^2-\umin^2$ on the whole $N$.
\end{lemma}

\begin{proof}
We recall from Subsection~\ref{sub:prelim} that the function $u$ is analytic and its critical level sets are discrete. It follows that there exists $\eta>0$ such that the level sets $\{u=\umin+\ep \}$ and $\{u=1/\ep\}$ are regular for any $0<\ep\leq\eta$.
For any $0<\ep\leq\eta$, let $N_\ep=N\cap \{\umin+\ep\leq u\leq 1/\ep\}$. We have $|\D u|^2-u^2=-\umin^2$ on ${\rm MIN}(u)$ and from the hypotesis 
$$
\limsup_{x\in N,x\to\infty}(|\D u|^2-u^2+\umin^2)\leq 0\,.
$$ 
In particular, for any $\delta>0$, there exists $\ep>0$ small enough so that $|\D u|^2-u^2+\umin^2\leq \delta$ on $\{u=1/\ep\}$. In fact, if this were not the case, it would exist a sequence $\{\ep_i\}_{i\in\N}$ of positive real numbers converging to zero such that for every $i\in\N$ there exists $p_i\in\{u=1/\ep_i\}$ with $(|\D u|^2-u^2+\umin^2)(p_i)> \delta$, and the superior limit of this sequence would be greater than $\delta$, in contradiction with the hypothesis.
We have thus proved that
\begin{equation}
\label{eq:shen_aux_aux_A}
\lim_{\ep\to 0^+}\max_{\pa N_\ep}(|\D u|^2-u^2+\umin^2)\,\leq\, 0\,.
\end{equation}
On the other hand, we can apply the Maximum Principle to~\eqref{eq:elliptic_A} inside $N_\ep$ for an arbitrarily small $\ep>0$, and using~\eqref{eq:shen_aux_aux_A} we find
$$
\max_{N}(|\D u|^2-u^2)\,=\,\lim_{\ep\to 0^+}\max_{N_\ep}(|\D u|^2-u^2)\,=\,\lim_{\ep\to 0^+}\max_{\pa N_\ep}(|\D u|^2-u^2)\,\leq\,-\umin^2\,.
$$
The thesis follows.
\end{proof}

\subsection{Proof of Theorem~\ref{thm:main_A}.}

The strategy of the proof of Theorem~\ref{thm:main_A} is completely analogue to the one employed in Section~\ref{sec:monoform} for the proof of~Theorem~\ref{thm:main_D}. For this reason, we will avoid to give some details, that can be easily recovered by the interested reader. First of all, we introduce the function $U:(\umin,+\infty)\rightarrow\R$ defined as
\begin{equation}
\label{eq:Up_A}
t \,\, \longmapsto \,\, U(t) \, = \, \Big(\frac{1}{t^2-\umin^2}\Big)^{\!\!\frac{n}{2}}\!\!\!\!\!\! \int\limits_{ \{ u = t \}\cap N} \!\!\!\!  |\D u| \, \rmd \sigma .
\end{equation}
Reasoning as in Subsection~\ref{sub:main_proof_D}, one sees that the function $U$ is well defined and
constant on the Anti de Sitter solution. Furthermore, now we prove that $U$ is always nondecreasing in $t$.

\begin{proposition}[Monotonicity, case $\Lambda<0$]
\label{pro:main_A}
Let $(M,\go,u)$ be a solution to problem~\eqref{eq:pb_A}. Let $N$ be a connected component of $M\setminus{\rm MIN}(u)$ and let $U$ be the function defined as in~\eqref{eq:Up_A}. If  
$$\liminf_{x\in N,\,x\to \infty}\left(u^2-\umin^2-|\D u|^2\right)\geq 0\,,
$$
then the function $U$ is monotonically nondecreasing.
\end{proposition}

\begin{proof}
Recalling $\De u=nu$, we easily compute
\begin{align}
\notag
{\rm div}\left[\frac{\D u}{(u^2-\umin^2)^{\frac{n}{2}}}\right]\,&=\,\frac{\De u}{(u^2-\umin^2)^{\frac{n}{2}}}-n\,u\,\frac{|\D u|^2}{(u^2-\umin^2)^{\frac{n}{2}+1}}
\\
\label{eq:div_Du_A}
\,&=\,\frac{n\,u}{(u^2-\umin^2)^{\frac{n}{2}+1}}\left(u^2-\umin^2-|\D u|^2\right)\,\geq\,0\,,
\end{align}
where the last inequality follows from Lemma~\ref{le:shen_aux_A}. Integrating by parts inequality~\eqref{eq:div_Du_A} in $\{t_1\leq u\leq t_2\}\cap N$ for some $t_1<t_2$, and applying the Divergence Theorem, we deduce
\begin{multline}
\label{eq:int_div_Du_A}
\int\limits_{\{u=t_1\}\cap N}\!\!\!\!\bigg\langle\frac{\D u}{(u^2-\umin^2)^{\frac{n}{2}}}\,\bigg|\,{\rm n}\bigg\rangle\,\rmd\sigma
\,\,\,+
\!\!\!\!\!\int\limits_{\{u=t_2\}\cap N}\!\!\!\!\bigg\langle\frac{\D u}{(u^2-\umin^2)^{\frac{n}{2}}}\,\bigg|\,{\rm n}\bigg\rangle\,\rmd\sigma
\,=\,
\\
=\!\!\!\!\!\int\limits_{\{t_1\leq u\leq  t_2\}\cap N}\frac{n\,u}{(u^2-\umin^2)^{\frac{n}{2}+1}}\left(u^2-\umin^2-|\D u|^2\right)\,\,\geq\,\,0\,,
\end{multline}
where ${\rm n}$ is the outer $\go$-unit normal to the set $\{t_1\leq u\leq t_2\}$. In particular, one has ${\rm n}=-\D u/|\D u|$ on $\{u=t_1\}$ and ${\rm n}=\D u/|\D u|$ on $\{u=t_2\}$. Therefore, formula~\eqref{eq:int_div_Du_A} rewrites as
$$
\int\limits_{\{u=t_2\}\cap N}\!\!\frac{|\D u|}{(u^2-\umin^2)^{\frac{n}{2}}}\,\rmd\sigma\,\,\,\geq \!\!\!\!\int\limits_{\{u=t_1\}\cap N}\!\!\frac{|\D u|}{(u^2-\umin^2)^{\frac{n}{2}}}\,\rmd\sigma
\,,
$$
which implies $U(t_2)\geq U(t_1)$, as wished.
\end{proof}

\noindent Combining Theorem~\ref{thm:main_A} with some approximations near the extremal points of the static potential $u$, we are able to characterize the set ${\rm MIN}(u)$ and to estimate the behavior of the $U(t)$'s as $t$ approaches $\umin$.

\begin{proposition}
\label{pro:estimate_A}
Let $(M,\go,u)$ be a solution to problem \eqref{eq:pb_A}.
Let $N$ be a connected component of $M\setminus{\rm MIN}(u)$ and let $U$ be the function defined by~\eqref{eq:Up_A}. If $\mathscr{H}^{n-1}\big({\rm MIN}(u)\cap \overline{N}\big)> 0$, then $\lim_{t\to \umin^+}U(t)=+\infty$. 
\end{proposition}

\begin{proof}
The proof is completely analogue to the proof of Propopsition~\ref{pro:estimate_D}. From the {\L}ojasiewicz inequality one deduces that there is a neighborhood $V$ of ${\rm MIN}(u)$ such that the inequality
\begin{equation}
\label{eq:loja_A}
|\D u|\,\geq\,c\,(u-\umin)
\end{equation}
holds on the whole $V$.
The second step is to rewrite $U(t)$ as
$$
U(t)
\,\,=\,\,
\bigg(\frac{1}{t^2-\umin^2}\bigg)^{\frac{n-2}{2}}\!\!\!\!\!\int\limits_{\{u=t\}\cap N}\!\! \bigg(\frac{|\D u|}{(u-\umin)(u+\umin)}\bigg) \,\rmd\sigma\,.
$$
Thanks to the compactness of $M$ and to the properness of $u$, for $t$ sufficiently close to $\umin$ we have $\{u=t\}\cap N\subset V$. For these values of $t$, using inequality~\eqref{eq:loja_A}, we have the following estimate
$$
U(t)
\,\,\geq\,\,\frac{c}{t+\umin}\,
\bigg(\frac{1}{t^2-\umin^2}\bigg)^{\frac{n-2}{2}}\!\!\!\!\!\cdot\,\,\big|\{u=t\}\cap N\big|\,.
$$
Proceeding exactly an in the proof of Proposition~\ref{pro:estimate_D}, one can show that
\begin{equation}
\label{eq:thesis_estimate_A}
\lim_{t\to \umin^+}\big|\{u=t\}\cap N\big| \, > \,0\,,
\end{equation}
and this concludes the proof.
\end{proof}

\noindent 
We are now in a position to prove Theorem~\ref{thm:main_A}, that we restate here, in an alternative -- but equivalent -- form, for reference.

\begin{theorem}
\label{thm:main_A_dritto}
Let $(M,\go,u)$ be a solution to problem~\eqref{eq:pb_A}, and let $N$ be a connected component of $M\setminus{\rm MIN}(u)$. 
Suppose that
	\begin{equation}
	\notag
	\liminf_{x\in N,\,x\to\infty}\left(u^2-\umin^2-|\D u|^2\right)(x)\,\,\geq\,\, 0\,.
	\end{equation}
Then, up to a normalization of $u$, the triple $(M,\go,u)$ is isometric to the Anti de Sitter solution~\eqref{eq:A}. In particular, $M\setminus{\rm MIN}(u)$ is connected and $M$ has a unique end.
\end{theorem}

\begin{proof}
On $N$, consider the function $U$ defined as in~\eqref{eq:Up_A} and fix $t_0\in (\umin,\infty)$. From Proposition~\ref{pro:main_A} we know that $U$ is nondecreasing, hence we have 
	$$
	\lim_{t\to \umin^+}U(t)\,\leq\,  U(t_0)\,=\,\Big(\frac{1}{t_0^2-\umin^2}\Big)^{\!\!\frac{n}{2}}\!\!\!\!\!\! \int\limits_{ \{ u = t_0 \}\cap N} \!\!\!\!  |\D u| \, \rmd \sigma\,<\,+\infty\,,
	$$
where in the latter inequality we have used the fact that $|\D u|$ is a continuous function and $\{u= t_0\}$ is compact (because $u$ is proper and $u\to +\infty$ at the infinity of $N$).
	Therefore, Proposition~\ref{pro:estimate_A} tells us that $\mathscr{H}^{n-1}\big({\rm MIN}(u)\cap \overline{N}\big)= 0$. This means that ${\rm MIN}(u)\cap \overline{N}$ cannot disconnect the manifold $M$, which in turn proves that $M\setminus{\rm MIN}(u)$ is connected. Therefore, we can apply Theorem~\ref{thm:shen_A} to deduce the thesis.
\end{proof}

\appendix

\section{Surface gravity}
\label{sec:surf_grav}

\subsection{Surface gravity of Killing horizons.}
\label{sub:surf_grav_KH} 
Let $(X,\gamma)$ be a $(n+1)$-dimensional vacuum spacetime, which means that $X$ is an $(n+1)$-dimensional manifold and $\gamma$ is a Lorentzian metric satisfying 
\begin{equation*}
\Ric_\gamma \,- \frac{\RRR_\gamma}{2} \, \gamma \,+ \Lambda \, \gamma\, = \, 0 \, , \quad \hbox{ in \,\, $X$} \, .
\end{equation*}
Suppose that there exists a Killing vector field $K$ on $X$, that is, a vector field such that $\mathcal{L}_K\gamma\,=\,0$ on the whole $X$.
A {\em Killing horizon} $\Sigma\subset X$ is a connected $n$-dimensional null hypersurface, invariant under the flow of $K$, such that 
$$
|K|^2_\gamma\,=\,0\, \ \hbox{ and }\ \, K\neq 0\,,\quad \hbox{on }\,\Sigma\,.
$$
Notice that $K$ is a null vector on $\Sigma$ by definition, and it is tangent to $\Sigma$ because of the invariance of $\Sigma$ under the flow of $K$.
Moreover, since $|K|^2_\gamma$ is constant on a Killing horizon $\Sigma$, the vector field $\na|K|^2_\gamma$, where $\na$ is the covariant derivative with respect to $\gamma$, is orthogonal to $\Sigma$. In particular it is also orthogonal to $K$. Since $\Sigma$ is a null hypersurface, it follows that $\na|K|^2_\gamma$ is a null vector. On the other hand, it is well known that two orthogonal null vectors are necessarily proportional to each other, and thus $\na|K|^2_\gamma$ must be proportional to $K$. In other words there exists a smooth function $\kappa \in \mathscr{C}^{\infty} (\Sigma)$ such that 
\begin{equation}
\label{eq:surf_grav}
2\,\kappa\,K_{|_{\Sigma}}\,=\,-(\na|K|^2_\gamma)_{|_{\Sigma}}\,.
\end{equation}
Using the basic properties of $K$ and $\Sigma$ (see for example~\cite{Bar_Car_Haw} or \cite[Theorem 7.1]{Heusler}), it is possible to prove that $\kappa$ is actually constant on $\Sigma$. {\em Once a normalization is chosen} for the Killing vector field $K$ in order to overcome the lack of scaling invariance of the above equation, it is usual to refer to the proportionality constant $\kappa \in \R$ as to the {\em surface gravity} of the Killing horizon $\Sigma$. In the case of static metrics,  the natural normalizations of the Killing vector field have been proposed in Subsection~\ref{sub:prelim}, at least in the most relevant situations. Finally, it is useful to recall (see for instance~\cite[Section~12.5]{Wald}) that the number $\kappa$ can also be computed through the following identity
\begin{equation}
\label{eq:surf_grav_2}
\kappa^2\,=\,-\frac{1}{2}(|\na K|_\gamma^2)_{|_{\Sigma}}\,.
\end{equation}
In the following subsection, we are going to take advantage of this fact.

\subsection{Surface gravity on the horizons of static spacetimes.} 
\label{sub:surf_grav_static}
Here we show that the general definition of surface gravity given in Subsection~\ref{sub:surf_grav_KH} above is coherent with the definition given in Subsection~\ref{sub:prelim} for static spacetimes.
We recall that a static triple $(M,\go,u)$ is a solution to problem~\eqref{eq:SES} such that $u=0$ on $\pa M$. We have already observed that any such triple gives rise to a static spacetime
$$
(X,\gamma)\,=\,\left(\R\times (M\setminus\pa M),-u^2 dt\otimes dt+\go\right)\,
$$
obeying the {\em vacuum Einstein field equations}.
%Its closure 
%$$
%(\overline{X},\bar\gamma)\,=\,(\R\times M,-u^2 dt\otimes dt+\go)
%$$
%is not a spacetime anymore, because $\bar\gamma$ is degenerate on $\R\times \pa M\subset \overline{X}$. 
In this case, there is a canonical choice of a Killing vector field, that is 
$$
K\,=\,\frac{\pa}{\pa t}\,.
$$

%{\color{blue}
The components of the boundary $\R\times \pa M$ can be interpreted as Killing horizons, as they are null hypersurfaces invariant under the flow of $K$, and such that $|K|_{\gamma}^2=-u^2=0$ on them. It is then natural to try to compute their surface gravities using the formul\ae\ of Subsection~\ref{sub:surf_grav_KH}. However, notice that the metric $\gamma$ becomes degenerate on the boundary, so, in order to use formula~\eqref{eq:surf_grav}, we would need to find new appropriate charts on the points of $\R\times \pa M$, with respect to which $\gamma$ remains Lorentzian.
Although this can be explicitly done in some special cases %\begin{comm}FORSE TOGLIERE LA PARENTESI(for instance, the Kruskal coordinates extend the Schwarzschild metric~\eqref{eq:S}, in such a way that the set $\{|x|=r_0(m)\}$ becomes a bifurcate Killing horizon, see for instance~\cite[Section~1.2.3]{Chrusciel_3}, and a similar construction can be done on any rotationally symmetric static spacetime~\cite{Walker})\end{comm}
, in general it seems quite an hard task.
On the other hand, formula~\eqref{eq:surf_grav_2} is far easier to apply. In fact, the metric $\gamma$ and the Killing vector being smooth, we can compute $|\na K|^2_\gamma$ on $X$, and then take the limit as we approach the boundary $\R\times \pa M$ to compute $\kappa$.
To this end, we introduce coordinates $(x^1,\dots,x^n)$ on an open set of $M$, so that $(x^0:=t,x^1,,\dots,x^n)$ is a set of coordinates on an open set of $X$. In the following computations, we will use greek letters for indices that vary between $0$ and $n$, and latin letters for indices that vary from $1$ to $n$.
With these conventions, one has $K^\a=\delta_0^\a$, whereas the Christoffel symbols of $\gamma$ satisfy
$$
\Gamma_{0\a}^\b\,=\,\frac{\gamma^{\b\eta}}{2}\big(\pa_0 \gamma_{\a\eta}+\pa_\a \gamma_{0\eta}-\pa_\eta \gamma_{0\a}\big)
\,=
\frac{\gamma^{\b\eta}}{2}\big(\delta_{0\eta}\pa_\a\, \gamma_{00}-\delta_{0\a}\,\pa_\eta \gamma_{00}\big)\,=\,
\begin{dcases}
\frac{\pa_\a u}{u} & \hbox{ if }\, \a\neq 0,\, \b=0\,,
\\
u\,\go^{\b\eta}\,\pa_\eta u\phantom{\frac{1}{u}}\!\!\! & \hbox{ if }\, \a= 0,\, \b\neq 0\,,
\\
0\phantom{\frac{\pa_i u}{u}} & \hbox{ otherwise}\,.
\end{dcases}
$$
Now we compute
$$
\na_\a K^\b\,=\,\pa_\a K^\b +\Gamma_{\a\eta}^\b K^\eta \,=\,\Gamma_{0\a}^\b\,,
$$
hence
\begin{align*}
|\na K|^2_{\gamma}\,&=\,\gamma^{\a\eta}\,\gamma_{\b\mu}\,\na_\a K^\b\,\na_\eta K^\mu
\\
&=\,\gamma^{\a\eta}\,\gamma_{\b\mu}\,\Gamma_{0\a}^\b\,\Gamma_{0\eta}^\mu
\\
&=\,\left[-\frac{1}{u^2}\,\go_{jq}\,(u\,\go^{jr}\,\pa_r u)\,(u\,\go^{qs}\,\pa_s u)-u^2\,\go^{ip}\,\frac{\pa_i u}{u}\,\frac{\pa_p u}{u}\right]
\\
&=\,-2\,|\D u|^2\,.
\end{align*}
If $\Sigma$ is a connected component of $\R\times\pa M$, taking the limit of formula~\eqref{eq:surf_grav_2} as we approach $\Sigma$, we obtain
\begin{equation}
\label{eq:surfacegravity_static}
\kappa\,=\,|\D u|_{|_\Sigma}\,,
\end{equation}
as expected.
%\begin{comm}
Formula~\eqref{eq:surfacegravity_static} justifies in some sense the canonical normalizations introduced in Subsection~\ref{sub:prelim}, that is, 
$\max_M u=1$ for $\Lambda>0$, $\sup_M u=1$ for $\Lambda=0$,
$\RRR^{\pa_\infty M}=-2(n-2)\Lambda/n$ if $\Lambda<0$. In fact, these normalizations are the ones under which the surface gravity of an horizon $\Sigma$, coincides precisely with $|\D u|_{|_\Sigma}$.
%QUEST'ULTIMO DISCORSO NON MI CONVINCE TROPPO, SAREI PER TOGLIERLO
%\end{comm}

\section{Further remarks in the case of negative cosmological constant}
\label{sec:appB}

%{\color{blue}
Since there is a huge amount of literature about the case $\Lambda<0$, we recall in this section the main definitions and known results, for the reader convenience.
%}

\subsection{Model solutions with nonspherical cross sections.}
\label{sub:nonspherical_models}

Here we collect, for completeness, some other interesting model solutions of~\eqref{eq:SES}, whose cross sections are not spheres.

\begin{itemize}

\item \underline{Schwarzschild--Anti de Sitter solutions with flat topology and mass $m>0$.}
\begin{align}
\label{eq:kottler_flat_A}
\nonumber 
\phantom{\qquad}M \, = \, [r_0(m),+\infty)\times\R^{n-1} \, ,  \qquad \go \, = \, \frac{dr\otimes dr}{r^2- 2m r^{2-n}}+r^2 g_{\R^{n-1}} \, , \\
u \, = \, \sqrt{r^2- 2m r^{2-n}} \, .\phantom{\qquad\qquad\qquad\qquad\qquad}
\end{align}
where $r_0(m)=(2m)^{1/n}$ is the positive solution of $r^2- 2m r^{2-n}=0$. We observe that the manifold $M$ is usually quotiented by a group of isometries in order to replace the second factor $\R^{n-1}$ with a torus, so that the boundary
$$
\pa M\,\,=\,\,\{r=r_0(m)\}\,.
$$
becomes compact. The metric $\go$ and the function $u$ extend smoothly to the boundary, and it holds
$$
|\D u|\,\,=\,\,(n-1)m^{\frac{1}{n}}\qquad\hbox{on}\ \, \pa M\,.
$$
The triple~\eqref{eq:SA} is conformally compact in the sense of Definition~\ref{def:CC_A} below, and the metric $u^{-2}\go$ induces the standard Euclidean metric $g_{\R^{n-1}}$ on the conformal infinity $\pa_\infty M$ (for the definition of conformal infinity see Subsection~\ref{sub:defi_app}, below Definition~\ref{def:CC_A}).
In particular, the scalar curvature $\RRR^{\pa_\infty M}$ of the metric induced by $u^{-2}\go$ on $\pa_\infty M$ is constant and equal to $0$. In this case, according to the discussion in Subsection~\ref{sub:prelim}, we have not a standard way of renormalizing $u$ in order to obtain an unambiguous notion of surface gravity.
 
Finally, the functions $u$ and $|\D u|$ go to $\infty$ as we approach the conformal infinity, and more precisely we have the following asymptotic behavior
$$
\lim_{r\to\infty}\left(u^2- \frac{ \RRR^{\pa_\infty M}}{(n-1)(n-2)} 
-|\D u|^2\right)\,\,=\,\,
\lim_{r\to\infty}\left(u^2-|\D u|^2\right)\,\,=\,\,0\,.
$$

\smallskip

\item \underline{Schwarzschild--Anti de Sitter solutions with hyperbolic topology and mass $m>-\mmax$.}
\begin{align}
\label{eq:kottler_hyp_A}
\nonumber 
\phantom{\qquad}M \, = \, [r_0(m),+\infty)\times\mathbb{H}^{n-1} \, ,  \qquad \go \, = \, \frac{dr\otimes dr}{-1+r^2- 2m r^{2-n}}+r^2 g_{\mathbb{H}^{n-1}} \, , \\
u \, = \, \sqrt{-1+r^2- 2mr^{2-n}} \, .\phantom{\qquad\qquad\qquad\qquad\qquad}
\end{align}
where $r_0(m)$ is the greatest positive solution of $-1+r^2-2mr^{2-n}=0$. We remark that, in order for such an $r_0(m)$ to exits, it is sufficient to set $m>-\mmax$, where $\mmax$ is defined as in~\eqref{eq:mmax_D}.
 In particular, negative masses are acceptable. 
As for the previous model solution, the second factor $\HH^{n-1}$ is usually quotiented in order to replace it with a compact hyperbolic manifold, so that the boundary
$$
\pa M\,\,=\,\,\{r=r_0(m)\}\,,
$$
becomes compact.
The triple~\eqref{eq:kottler_hyp_A} is conformally compact in the sense of Definition~\ref{def:CC_A} below, and the metric $u^{-2}\go$ induces the standard hyperbolic metric $g_{\mathbb{H}^{n-1}}$ on the conformal infinity $\pa_\infty M$ (for the definition of conformal infinity see Subsection~\ref{sub:defi_app}, below Definition~\ref{def:CC_A}).
In particular, the scalar curvature $\RRR^{\pa_\infty M}$ of the metric induced by $u^{-2}\go$ on $\pa_\infty M$ is constant and equal to $-(n-1)(n-2)=2(n-2)\Lambda/n$, hence, according to the definition given in Subsection~\ref{sub:prelim}, the surface gravity of the horizon $\pa M$ can be computed as
$$
|\D u|_{|_{\pa M}}\,\,=\,\,r(m)\left[1+\frac{(n-2)m}{r^{n}(m)}\right]\,.
$$
Finally, the quantities $u$ and $|\D u|$ obey the following asymptotic behavior
$$
\lim_{r\to\infty}\left(u^2- \frac{ \RRR^{\pa_\infty M}}{(n-1)(n-2)} 
-|\D u|^2\right)\,\,=\,\,0\,.
$$

\smallskip

\item \underline{Anti Nariai solution (Complete non-compact Cylinder).}
\begin{align}
\label{eq:cylsol_A}
\nonumber 
\phantom{\qquad\qquad}M \, = \, (-\infty, +\infty)\times\mathbb{H}^{n-1}\,, \qquad  \go \, = \, \frac{1}{n}\,\big[dr\otimes dr+(n-2)\,g_{\mathbb{H}^{n-1}}\big]\,, \\
u \, = \, \cosh (r) \, .\phantom{\qquad\qquad\qquad\qquad\qquad\qquad\quad}
\end{align}
The Anti Nariai solution has empty boundary and the set 
$$
{\rm MIN}(u)\,=\,\{p\in M\,:\, u(p)=1\}\,,
$$ 
coincides with the sphere $\{0\}\times\mathbb{H}^{n-1}$. Moreover, this solution has two ends, where the function $u$ goes to infinity. Finally, we have
$$
u^2-1-\frac{1}{n}|\D u|^2\,\,=\,\,0\,,
$$
pointwise on $M$ and, in particular, in contrast with the (Schwarzschild--)Anti de Sitter solutions, it holds $\lim_{r\to\infty}(u^2-\umin^2-|\D u|^2)=-\infty$. 
 %However, there are some important distinctions, in fact the Anti Nariai triple is not conformally compact, and its set ${\rm MIN}(u)$ is a whole slice.
\end{itemize}
%AGGIUNGERE UN COMMENTO DEL TIPO: "The Anti de Sitter triple~\eqref{eq:D} and the ANTI Nariai triple~\eqref{eq:cylsol_D} can be intepreted as the limit cases of the Schwarzschild-- Anti de Sitter solution when $m\to 0^+$ and $m\to???????$, respectively" MA UN PO' PIU' CIRCOSTANZIATO.
%

\subsection{Standard definitions and known results.}
\label{sub:defi_app}

In this subsection, we discuss some of the classical results on static metrics with negative cosmological constant, which we recall are triples $(M,\go,u)$ that satisfy the following problem
\begin{equation}
\label{eq:pb_conbordo_A}
\begin{dcases}
u\,\Ric=\DD u-n\,u\,\go, & \mbox{in } M\\
\ \;\,\De u=n\, u, & \mbox{in } M\\
\ \ \ \ \; u>0, & \mbox{in }  M \\
\ \ \ \ \; u=0, & \mbox{on }  \pa M \\
\  u(x)\rightarrow +\infty & \mbox{as } x\rightarrow \infty
\end{dcases}
 \qquad  \hbox{with} \quad \RRR\equiv -n(n-1)\, .
\end{equation}
This is precisely system~\eqref{eq:pb_A}, that we have rewritten here for the sake of reference.
However, notice that this time we are not assuming that $\pa M$ is empty. In fact, most of the following classical results (with the important exception of Theorem~\ref{thm:uniquenessW_A}) do not rely on this hypotesis.

The usual way to obtain classification results for solutions $(M,\go,u)$ of system~\eqref{eq:pb_conbordo_A} is to ask for some suitable asymptotic behavior of the triple. This approach was started in~\cite{Penrose}, where the definition of conformal compactness was introduced, and developed by a number of authors, see for instance~\cite{Chr_Sim,Hij_Mon} and the references therein. Below we will retrace this approach and we will discuss some of the results that one can obtain from it.

\begin{defapp}
\label{def:CC_A}
A triple $(M,\go,u)$ that solves problem~\eqref{eq:pb_conbordo_A} is said to be {\em conformally compact} if
\begin{itemize}
		\item[(i)] The manifold $M\setminus\pa M$ is diffeomorphic to the interior of a compact manifold $\overline{M}_\infty$ with boundary $\pa \overline{M}_\infty=\pa M\cup\pa_{\infty}M$, with $\pa_\infty M\cap M=\emptyset$,
		
\item[(ii)] The metric $\bar g=u^{-2} \go$ extends smoothly to the whole $\overline{M}_\infty\setminus\pa M=(M\setminus\pa M)\cup\pa_\infty M$. 
	\end{itemize}
\end{defapp}

\noindent The manifold $\pa_\infty M$ is usually called the {\em conformal infinity} of the conformally compact triple $(M,\go,u)$.
Each connected component of $\pa_\infty M$ corresponds to an end of the manifold $M$. 
A standard computation (see for instance~\cite{Graham,Mazzeo}) shows that, if $(M,\go,u)$ is conformally compact, then the Riemannian tensor of $\go$ satisfies
$$
\RRR_{ijkl}\,=\,-|d(u^{-1})|_{\bar g}^2\Big[\go_{ik}\,\go_{jl}-\go_{il}\,\go_{jk}\Big]+\mathcal{O}(u^{-3})\,,
$$
as $u\to\infty$.
In particular, since the scalar curvature of $\go$ is constant and equal to $-n(n-1)$, it follows that $|d(u^{-1})|_{\bar g}$ goes to $1$ as we approach the infinity. Therefore, the sectional curvature of $g$ converge to $-1$, and the manifold $(M,g)$ is weakly asymptotically hyperbolic in the sense of~\cite{Wang_1}.

%This justifies the following definition.
%
%
%
%\begin{defapp}
%\label{def:ALAdS}
%A triple $(M,\go,u)$ that solves problem~\eqref{eq:pb_conbordo_A} is said to be {\em asymptotically locally Anti de Sitter} if it is conformally compact and $|d(u^{-1})|_{\bar g}=1$ on the whole $\pa_\infty M$.
%\end{defapp}

As already observed, the conformal infinity $\pa_\infty M$ may have more than one connected component, each corresponding to a different end of $M$. However, the next result shows that, under suitable hypoteses, the conformal infinity is forced to be connected.

\begin{proposition}[{\cite[Theorem~I.1]{Chr_Sim}, \cite[Proposition~2]{Hij_Mon}}]
	\label{pro:CC_A}
	Let $(M,\go,u)$ be a conformally compact solution to problem~\eqref{eq:pb_conbordo_A}. If 
	\begin{itemize}
		\item either $n=3$,
		\item or the metric $(u^{-2}\go)_{|_{\pa_\infty M}}$ has nonnegative scalar curvature and $\pa M=\emptyset$,
	\end{itemize}
	then the conformal infinity $\pa_\infty M$ is connected.
\end{proposition}

\begin{remark}
It may be interesting to compare the second point in Proposition~\ref{pro:CC_A} with Corollary~\ref{cor:consequences_ALAdS}. In fact, this latter result tells us that, if one requires that a stronger bound on the scalar curvature of $(u^{-2}\go)_{|_{\pa_\infty M}}$ holds on the ends of a single region $N\subset M\setminus{\rm MIN}(u)$, then not only the conformal infinity is connected, but also the solution is isometric to the Anti de Sitter solution~\eqref{eq:A}.
\end{remark}

\noindent 
Another important property of conformally compact triples is that it is possible to expand the metric $\go$ and the potential $u$ in terms of the so called {\em special defining function} of the conformal infinity. We refer to~\cite{Hij_Mon} for the details. Here we just remark that, as a consequence of~\cite[Lemma~3]{Hij_Mon}, one easily computes the following asymptotic behavior of the potential $u$
\begin{equation}
\label{eq:ALAdS_A}
\lim_{u\to+\infty}\left[u^2-\frac{\RRR^{\pa_\infty M}}{(n-1)(n-2)}-|\D u|^2\right]\,=\,0\,,
\end{equation}
where $\RRR^{\pa_\infty M}$ is the scalar curvature of $(\pa_\infty M,\bar g_{|_{\pa_\infty M}})$. 
This is a good occasion to notice that the Anti de Sitter triple~\eqref{eq:A} is indeed conformally compact, and its conformal infinity is isometric to the unit round sphere, so that the right hand side of~\eqref{eq:ALAdS_A} is equal to $1$.
 This observation suggests to introduce the following definition, which is a stronger version of Definition~\ref{def:CC_A}.

\begin{defapp}
\label{def:AAdS}
A triple $(M,\go,u)$ that solves problem~\eqref{eq:pb_conbordo_A} is said to be {\em asymptotically Anti de Sitter} if it is conformally compact, the conformal boundary $\pa_\infty M$ is diffeomorphic to a sphere $\Sph^{n-1}$ and the metric $u^{-2}\go$ extends to the standard spherical metric $g_{\Sph^{n-1}}$ on $\pa_\infty M$.
\end{defapp}

\noindent From formula~\eqref{eq:ALAdS_A} we immediately deduce that, if $(M,\go,u)$ is asymptotically Anti de Sitter, it holds 
\begin{equation*}
%\label{eq:AAdS_A}
\lim_{u\to +\infty} \left(u^2-1-|\D u|^2\right)=0\,,
\end{equation*}
and this refined version of~\eqref{eq:ALAdS_A} has been used to prove Corollary~\ref{cor:small_wang_A}.
Furthermore, asymptotically Anti de Sitter solutions can be seen to be asymptotically hyperbolic in the sense of~\cite[Definition~2.3]{Wang_1}, and for this class of manifolds a notion of mass has been defined by Wang~\cite{Wang_1}. Under the hypotesis of the existence of a spin structure, this mass satisfies a Positive Mass Theorem (see~\cite[Theorem~2.5]{Wang_1}). More precisely, the mass of a solution $(M,\go,u)$ of~\eqref{eq:pb_conbordo_A}, such that $(M,\go)$ is spin and has empty boundary, is nonnegative and it is zero if and only if $(M,\go)$ is isometric to an hyperbolic space form. As a consequence of this rigidity statement, one deduces the following  uniqueness result for the Anti de Sitter triple, which generalizes a classical result in~\cite{Bou_Gib_Hor}.

\begin{theorem}[{\cite[Theorem~1]{Wang_2}}]
	\label{thm:uniquenessW_A}
	Let $(M,\go,u)$ be an asymptotically Anti de Sitter solution to problem~\eqref{eq:pb_conbordo_A}, and suppose that $\pa M=\emptyset$. If $M$ is spin, then $(M,\go,u)$ is isometric to the Anti de Sitter triple~\eqref{eq:A}.
\end{theorem}

\noindent This result has been further extended by Qing in~\cite{Qing}, who was able to drop the spin assumption. The idea of~\cite{Qing} is to glue an asymptotically flat end to the conformally compactified manifold, so that it is possible to use the rigidity statement of the Positive Mass Theorem 
%of Schoen-Yau~\cite{Sch_Yau,Sch_Yau_2}, more precisely, the version 
with corners proved by Miao~\cite{Miao}. 
Another extension of Theorem~\ref{thm:uniquenessW_A}, stated in~\cite[Theorem~8]{Hij_Mon}, allows for a different topology of the conformal infinity, provided that an opportune spinor field exists on the conformal boundary.

For completeness, we mention that another version of mass has been given by Zhang~\cite{Zhang} in the three dimensional case. Moreover, we also point out that a more general definition of mass has been provided by Chrusciel and Herzlich in~\cite{Chr_Her}, and as a consequence another proof of Theorem~\ref{thm:uniquenessW_A} has been provided (see~\cite[Theorem~4.3]{Chr_Her}).

It is important to remark that, when our static triple is not asymptotically Anti de Sitter, in general the mass introduced by Wang~\cite{Wang_1} and Chrusciel-Herzlich~\cite{Chr_Her} is not defined.
For this reason, one is led to investigate other possible definitions of mass which allow for less rigid behaviours at infinity, while preserving some interesting properties.

We cite one of the most important of these alternative masses. Given a $3$-dimensional solution $(M,\go,u)$ to problem~\eqref{eq:pb_conbordo_A} and a closed compact surface $\Sigma$ in $M$, the {\em Hawking mass} of $\Sigma$ is defined as
$$
m_H(\Sigma)\,=\,\sqrt{\frac{|\Sigma|}{16\pi}}\,\left[1-{\rm genus}(\Sigma)-\frac{1}{16\pi}\int_\Sigma(\HHH^2-4)\,\rmd\sigma\right]\,,
$$
where $\HHH$ is the mean curvature of $\Sigma$. The following result by Chrusciel and Simon compares the Hawking mass with another definition of mass, which is very much in the spirit of the virtual mass that we have defined in Subsection~\ref{sub:surfmass} for the case $\Lambda>0$.

\begin{theorem}[{\cite[Theorem~I.5]{Chr_Sim}}]
\label{thm:chr_sim}
Let $(M,\go,u)$ be a $3$-dimensional solution of system~\eqref{eq:pb_conbordo_A} with nonempty boundary $\pa M$.
Suppose that $(M,\go,u)$ is conformally compact, that the conformal infinity $\pa_\infty M$ (which is connected thanks to Proposition~\ref{pro:CC_A}) satisfies ${\rm genus}(\pa_\infty M)\geq 2$, and that the scalar curvature $\RRR^{\pa_\infty M}$ induced by $u^{-2}\go$ on $\pa_\infty M$ is constant and equal to $-6$. Suppose further that 
$$
0\leq \kappa:=\max_{\pa M} |\D u| \leq 1\,.
$$
Then there is a unique value $\mu=\mu(M,\go,u)\leq 0$ such that the boundary of the Schwarzschild--Anti de Sitter triple with hyperbolic topology~\eqref{eq:kottler_hyp_A} and mass $\mu$ has surface gravity equal to $\kappa$, and it holds
\begin{equation}
\label{eq:chr_sim}
m_H(\{u=t\})\,\leq \,\mu\,,
\end{equation}
for all $t$.
Moreover,  if the equality $m_H(\{u=t\})=\mu$ holds for some $t$, then, up to a normalization of $u$, the triple $(M,\go,u)$ is isometric to the Schwarzschild--Anti de Sitter solution with hyperbolic topology~\eqref{eq:kottler_hyp_A}.
\end{theorem}

\noindent
We point out that the parameter $\mu$ in the above theorem is the natural analogue of the virtual mass that we have defined in Subsection~\ref{sub:surfmass} in the case $\Lambda>0$. In fact, it is obtained in the same way, that is, by computing the maximal surface gravity of the horizons and then comparing it with the model solutions, which in this case are the Schwarzschild--Anti de Sitter solutions with hyperbolic topology.
Formula~\eqref{eq:chr_sim} shows a connection between this quantity $\mu$ and the Hawking mass. 
Moreover, as it is discussed below, the parameter $\mu$ plays an important role in the proof of an area bound and a Black Hole Uniqueness Theorem for Schwarzschild--Anti de Sitter solutions with hyperbolic topology in dimension $n=3$. These results are in line with the ones that will be proven for solutions with $\Lambda>0$ in the forthcoming paper~\cite{Bor_Maz_2-II}.
We stress that the hypotesis $\mu\leq 0$ (or equivalently $\kappa\leq 1$) in Theorem~\ref{thm:chr_sim} above is necessary. In fact, the crucial step of the proof is the application of the Maximum Principle to the differential inequality~\cite[formula~(V.4)]{Bei_Sim}, which is elliptic only if $\mu\leq 0$.

Under the same hypoteses of Theorem~\ref{thm:chr_sim} and with the same notations, Chrusciel and Simon also prove that, for any boundary component $\Sigma\subset \pa M$ with maximal surface gravity $\kappa$, it holds
\begin{equation}
\label{eq:areabound_chr_sim}
|\Sigma|\,\geq\,\frac{{\rm genus}(\Sigma)-1}{{\rm genus}(\pa_\infty M)-1}\,4\pi\,r^2(\mu)\,,
\end{equation}
where $\mu$ is again the parameter corresponding to the mass of the Schwarzschild--Anti de Sitter solution with hyperbolic topology whose horizon has surface gravity $\kappa$, and $r(\mu)$ is the largest positive solution of $1-x^2+2\mu/x=0$.
Building on this, Lee and Neves proved the following Black Hole  Uniqueness Theorem for Schwarzschild--Anti de Sitter solutions with hyperbolic topology~\eqref{eq:kottler_hyp_A} and with nonpositive virtual mass.

\begin{theorem}[{\cite[Theorem~2.1]{Lee_Nev}}]
\label{thm:BHU_LeeNeves_A}
In the same hypoteses and notations of Theorem~\ref{thm:chr_sim}, suppose that there exists an horizon $\Sigma\subset\pa M$ with maximal surface gravity $\kappa$ and with 
$$
{\rm genus}(\Sigma)\,\geq\, {\rm genus}(\pa_\infty M)\,.
$$
Then, up to a normalization of $u$,
$(M,\go,u)$ is isometric to the Schwarzschild--Anti de Sitter solution with hyperbolic topology~\eqref{eq:kottler_hyp_A} and with mass $\mu$, where $\mu$ is the mass of the model solution~\eqref{eq:kottler_hyp_A} whose horizon has surface gravity $\kappa$.
\end{theorem} 

\noindent
The proof of this theorem is based on the monotonicity of the Hawking mass under inverse mean curvature flow, in the spirit of the classical work of Huisken-Hilmanen~\cite{Hui_Ilm} in the asymptotically flat case. This monotonicity is used to prove a bound from above on $|\Sigma|$, which combined with inequality~\eqref{eq:areabound_chr_sim}, recalling from~\eqref{eq:chr_sim} that the Hawking mass is controlled by $\mu$, gives the thesis.

Theorem~\ref{thm:BHU_LeeNeves_A} provides a first uniqueness result for static black holes with negative cosmological constant. However, as already noticed, this result only works for solutions with negative mass and hyperbolic topology. Unfortunately, it seems that no characterization is available in literature for what should be the most natural model, that is, the Schwarzschild--Anti de Sitter solution with spherical topology~\eqref{eq:SA}.

%%%%%%%%%%%%%%%%%%%%%%%%%%%%%%%%%%%%%%%%%%%%%%%
%%%%%%%%%%%%%%%%%%%%%%%%%%%%%%%%%%%%%%%%%%%%%%%

\subsection*{Acknowledgements}
{\em The authors would like to thank A. Carlotto and P. T. Chru\'sciel for their interest in our work and for stimulating discussions during the preparation of the manuscript. The authors are members of the Gruppo Nazionale per l'Analisi Matematica, la Probabilit\`a e le loro Applicazioni (GNAMPA) of the Istituto Nazionale di Alta Matematica (INdAM) and are partially founded by the GNAMPA Project ``Principi di fattorizzazione, formule di monotonia e disuguaglianze geometriche''. The paper was partially completed during the authors' attendance to the program ``Geometry and relativity'' organized by the Erwin Schr\"{o}dinger International Institute for Mathematics and Physics (ESI).
}

%%%%%%%%%%%%%%%%%%%%%%%%%%%%%%%%%%%%%%%%%%%%%%%
%%%%%%%%%%%%%%%%%%%%%%%%%%%%%%%%%%%%%%%%%%%%%%%

\bibliographystyle{plain}
\bibliography{biblio}

\end{document}